\documentclass[12pt]{article}%
\usepackage{amsmath,amssymb,eucal}
\usepackage{amsfonts}
\usepackage{mathrsfs}
\usepackage{slashed}
\numberwithin{equation}{section} 
\def\beq{\begin{equation}}
\def\eeq{\end{equation}}

\def\bC{ {{\mathbb{C}}}}
\def\bR{ {{\mathbb{R}}}}

\def\bp{ {\mathbf{p}} }

\def\fd{ {\mathfrak{d}} }
\newcommand{\dB}{{\mathbb{B}}}

\newtheorem{defn}{{\bf Definition}}[section]
\newtheorem{thm}[defn]{{\bf Theorem}}
\newtheorem{cor}[defn]{{\bf Corollary}}
\newtheorem{lem}[defn]{{\bf Lemma}}
\newtheorem{prop}[defn]{{\bf Proposition}}
\newtheorem{rem}[defn]{{\bf Remark}}

\newtheorem{notation}[defn]{Notation}
\newenvironment{proof}[1][Proof]{\textbf{#1.} }{\hfill \rule{0.5em}{0.5em}}

\begin{document}

\title{Abstract Wiener measure using abelian Yang-Mills action on $\mathbb{R}^4$}
\author{Adrian P. C. Lim \\
Email: ppcube@gmail.com
}

\date{}

\maketitle

\begin{abstract}
Let $\mathfrak{g}$ be the Lie algebra of a compact Lie group. For a $\mathfrak{g}$-valued 1-form $A$, consider the Yang-Mills action \begin{equation} S_{{\rm YM}}(A) = \int_{\mathbb{R}^4}  \left|dA + A \wedge A \right|^2\ d\omega \nonumber
\end{equation}
using the Euclidean metric on $T\mathbb{R}^4$. When we consider the Lie group ${\rm U}(1)$, the Lie algebra $\mathfrak{g}$ is isomorphic to $\mathbb{R} \otimes i$, thus $A \wedge A = 0$. For a simple closed loop $C$, we want to make sense of the following path integral,
\begin{equation}
\frac{1}{Z}\ \int_{A \in \mathcal{A} /\mathcal{G}} \exp \left[  \int_{C} A\right] e^{-\frac{1}{2}\int_{\mathbb{R}^4}|dA|^2\ d\omega}\ DA, \nonumber
\end{equation}
whereby $DA$ is some Lebesgue type of measure on the space $\mathcal{A} /\mathcal{G}$ containing $\mathfrak{g}$-valued 1-forms modulo gauge transformations, and $Z$ is some partition function.

We will construct an Abstract Wiener space for which we can define the above Yang-Mills path integral rigorously, applying renormalization techniques found in lattice gauge theory. We will further show that the Area Law formula does not hold in the abelian Yang-Mills theory.
\end{abstract}

\hspace{.35cm}{\small {\bf MSC} 2020: 81T13, 81T08} \\
\indent \hspace{.35cm}{\small {\bf Keywords}: Yang-Mills, Wiener measure, axial gauge fixing, abelian, \\
\indent \hspace{2.5cm} renormalization}




\section{Preliminaries}\label{s.pre}

The main goal of this article is to define a probability measure using the Yang-Mills action for an abelian gauge group, over in $\bR^4$. This is in preparation for its sequel, which we will consider the non-abelian case. The construction is using the Abstract Wiener space formalism as developed by Gross, hence is different from the construction via lattice gauge approximation. But we will have to adopt and apply renormalization techniques used in lattice gauge theory, as described in the papers written by Balaban. This construction will then be used in a sequel to this article, whereby we will define a probability measure using the Yang-Mills action in the case of a non-abelian gauge group.

This article is divided into two parts. The first part is a brief discussion on lattice gauge theory as done by Balaban. It is not our aim to give a comprehensive account of lattice gauge theory here. Instead, we will highlight certain renormalization ideas used in lattice gauge theory, which we will need to apply later on. The second half of this article starts from Section \ref{s.gm}, whereby we outline the Abstract Wiener space construction, as given in \cite{CSLim01}.

We are aware that in nature, there are no confining quarks in electromagnetism. Guth in \cite{PhysRevD.21.2291} has argued heuristically that for weak coupling, ${\rm U}(1)$ lattice gauge theory does not satisfy the Area Law formula. A rigorous proof of Guth's result is given in \cite{frohlich1982}, which is set on a discrete lattice $\mathbb{Z}^4 \subset \bR^4$. We are also aware that asymptotic freedom does not hold for abelian Yang-Mills theory, which describes the electromagnetic force.

After the construction of a probability space, we will define the Yang-Mills path integral in Section \ref{s.fi} using renormalization techniques. We will further show that the Area Law formula does not hold for the abelian Yang-Mills theory.

Our ambient space is $\bR^4$ equipped with coordinate axes, with corresponding coordinates $(x^0, x^1, x^2, x^3)$, and choose the standard Riemannian metric (Euclidean metric) on $T\mathbb{R}^4$. We will call $x^0$ as the time coordinate. Now let $T^\ast\bR^4$ denote the trivial cotangent bundle over $\bR^4$, i.e. $T^\ast \bR^4 \cong \bR^4 \times \Lambda^1(\bR^4)$ and $\Lambda^1(\bR^3)$ denote the subspace in $\Lambda^1(\bR^4)$ spanned by $\{dx^1, dx^2, dx^3\}$.

We can also define $\Lambda^2(T^\ast \bR^4)$, the trivial 2-th exterior power of the cotangent bundle over $\bR^4$. Note that $\Lambda^2(T^\ast \bR^4) \cong \bR^4 \times \Lambda^2(\bR^4)$ and it inherits an inner product $\langle \cdot, \cdot \rangle_2$ from the Riemannian metric, thus \beq \{dx^1 \wedge dx^2, dx^3 \wedge dx^1, dx^1 \wedge dx^2, dx^0 \wedge dx^1, dx^0 \wedge dx^2, dx^0 \wedge dx^3\} \nonumber \eeq is an orthonormal basis in $\Lambda^2 (T^\ast\bR^4)$. See \cite{MR1312606}.

Let $P \rightarrow \bR^4$ be a trivial complex line bundle, with structure group ${\rm U}(1)$. The Lie algebra $\mathfrak{u}(1)$ will be isomorphic to $\bR \otimes i$. Denote the group of all smooth ${\rm U}(1)$-valued mappings on $\bR^4$ by $\mathcal{G}$, called the gauge group. This gauge group consists of elements of the form $e^{i\alpha(x)}$, $\alpha$ is a smooth continuous real-valued function on $\bR^4$. For any complex scalar field $\psi(x)$, the gauge group acts on it via $ \psi(x) \mapsto e^{i\alpha(x)}\psi(x)$, whereby $x \in \bR^4$. See \cite{MR1031992}.

A connection is an $\mathfrak{u}(1)$-valued 1-form, written as $A\otimes i = \sum_{\mu=0}^3 A_\mu \otimes dx^\mu\otimes i$, and it defines a covariant derivative \beq D_\mu \psi(x) = \partial_\mu \psi(x) + A_\mu(x)\psi(x),\ \partial_\mu \equiv \frac{\partial }{\partial x^\mu}. \nonumber \eeq Each $A \otimes i = \sum_{\mu=0}^3 A_\mu \otimes dx^\mu \otimes i$ transforms under gauge transformation as follows: \beq A_\mu(x)\otimes i \mapsto [A_\mu(x) + \partial_\mu\alpha(x)]\otimes i. \label{e.g.1} \eeq

To ease the notation, we will in future drop the $i\equiv \sqrt{-1}$ from the connection. The vector space of all smooth connections will be denoted by $\mathcal{A}$.  The orbit of an element $A \in \mathcal{A}$ under this gauge transformation will be denoted by $[A]$ and the set of all orbits by $\mathcal{A}/\mathcal{G}$.

A volume form on $\bR^4$ is given by $d\omega \equiv dx^0 \wedge dx^1 \wedge dx^2 \wedge dx^3$. This allows us to define a Hodge star operator $\ast$ acting on $2$-forms, $\ast: \Lambda^2(T^\ast \bR^4) \rightarrow \Lambda^{2}(T^\ast \bR^4)$ such that for $u, v \in \Lambda^2(T^\ast \bR^4)$, we have \beq u \wedge \ast v = \langle u, v \rangle_2\ d\omega. \label{e.x.7} \eeq  An inner product on the set of smooth sections of $\Lambda^2(T^\ast \bR^4)$, $\Gamma(\Lambda^2(T^\ast \bR^4))$, is then defined as \beq \langle u, v \rangle := \int_{\bR^4} u \wedge \ast v \equiv \int_{\bR^4} \langle u, v \rangle_2\ d\omega. \nonumber \eeq

For $A \in \mathcal{A}$, the abelian Yang-Mills action is given by \beq S_{{\rm YM}}(A) = \langle dA, dA \rangle \equiv \int_{\bR^4}  \langle dA, dA\rangle_2\ d\omega. \label{e.c.1} \eeq  Note that this action is invariant under gauge transformations, given by Equation (\ref{e.g.1}).

Let $C$ be a simple closed curve in the manifold $\bR^4$. It is of interest to make sense of the following abelian Yang-Mills path integral,
\beq \frac{1}{Z}\int_{A \in \mathcal{A} /\mathcal{G}} \exp \left[  \int_{C} A\otimes i\right] e^{-\frac{1}{2}S_{{\rm YM}}(A)}\ DA, \label{e.ym.1} \eeq whereby $DA$ is some Lebesgue type of measure on the space of connections, modulo gauge transformations and \beq Z = \int_{A \in \mathcal{A} /\mathcal{G}}e^{-\frac{1}{2}S_{{\rm YM}}(A)}\ DA. \nonumber \eeq

Using axial gauge fixing, every orbit $[A] \in \mathcal{A} /\mathcal{G}$ can be gauge transformed into the form $A = \sum_{j=1}^3a_{j} \otimes  dx^j $,  subject to the conditions \beq a_{1}(0, x^1, 0, 0) = 0,\ a_{2}(0, x^1, x^2, 0)=0,\ a_{3}(0, x^1, x^2, x^3) = 0. \nonumber \eeq We will drop these restrictions and hence consider the Yang-Mills integral in Expression \ref{e.ym.1} to be over the space of smooth 1-forms of the form $A =  \sum_{j=1}^3a_{j} \otimes  dx^j $, whereby $a_{j}: \bR^4 \rightarrow \bR$ is smooth and it decays to zero as $| (x^0, x^1, x^2, x^3)| \rightarrow \infty$.

Its differential is given by
\begin{align}
dA
=&  \sum_{j=1}^3  a_{0:j} \otimes dx^0 \wedge dx^j + \sum_{1\leq i<j \leq 3}a_{i;j} \otimes  dx^i \wedge dx^j , \nonumber
\end{align}
for $a_{i;j} := \partial_i a_{j} - \partial_j a_{i} \equiv \frac{\partial}{\partial x^i}a_j - \frac{\partial}{\partial x^j}a_i$, $a_{0:j} := \partial_0 a_{j} \equiv \frac{\partial}{\partial x^0}a_j$. Hence, \beq |dA|^2 := S_{{\rm YM}}(A) =  \int_{\bR^4} \sum_{j=1}^3|a_{0:j}|^3 + \sum_{1\leq i<j \leq 3}|a_{i;j}|^2\ d\omega. \label{e.ym.3} \eeq

We want to make sense of Expression \ref{e.ym.1}, which is now of the form \beq \frac{1}{Z}\int_A e^{\int_C \sum_{j=1}^3 A_j \otimes dx^j \otimes i}e^{-|dA|^2/2} DA,\ i = \sqrt{-1}. \label{e.ym.2} \eeq

\begin{rem}\label{r.ym.1}
\begin{enumerate}
  \item Note that we impose axial gauge fixing first, before defining a path integral, which we will compute the propagator later.
  \item Axial gauge fixing is used here, to define a non-degenerate inner product using the differential $d$, because we only consider the space of smooth 1-forms, which decay to zero at infinity. In \cite{rivasseau01} however, the authors only consider $\langle \partial_t \omega, \partial_t \omega\rangle$, which is degenerate. This is because the authors considered a compact ambient space, instead of in $\bR^4$.
\end{enumerate}
\end{rem}

The usual approach to define the path integral is via finite dimensional lattice gauge approximation, explained in \cite{glimm1981quantum}. But other authors have used different approaches besides using lattice approximation. The authors in \cite{2014arXiv1406.4177F} and \cite{Levy} applied white noise analysis to define a Yang-Mills Gaussian measure. The authors in \cite{GROSS198965} used stochastic calculus to define a 2-dimensional Yang-Mills measure. Also stated in \cite{douglas01}, it is not necessary to consider the lattice approximation when defining a quantum Yang-Mills theory. We also refer the reader to an article written by Chatterjee, whereby he gave references to other people who had attempted to define a Yang-Mills probability measure. See \cite{2018arXiv180301950C}.

\section{Balaban's approach to Renormalization}

Balaban did some impressive work from 1984 to 1989, by establishing ultraviolet stability in three and four-dimensional non-abelian lattice gauge theories, in a series of three papers \cite{volumeI, volumeII, volumeIII} published in 1982-83; two papers \cite{propagator1, propagator2} published in 1984; a series of five papers \cite{1985a, 1985b, 1985c, 1985d, 1985e} all published in 1985; and a series of five papers \cite{1987-99_RG-01, 1987-99_RG-02, 1987-99_RG-03, 1987-99_RG-04, 1987-99_RG-05} published from 1987 to 1989. In \cite{1985a, 1985b, 1985c, 1985d, 1985e}, Balaban considered the Yang-Mills action, but for dimension $d=3$. As remarked in \cite{douglas01}, one would need to introduce renormalization, if one uses lattice gauge approximation to define the functional integral.

To describe Balaban's approach to renormalization, we consider $d = 1$ for simplicity. He fixed a number $L > 1$ and consider a bounded lattice $[-L, L]$, with unit spacing. He partitioned it into a lattice $L^{-k}[-L, L]$, whereby each spacing will be $\epsilon = L^{-k}$. Each node in the lattice is considered as independent from its neighbors and he constructed a $2L^{k+1}+1$-dimensional integral from the lattice. By rescaling the lattice $L^{-k}[-L, L]$ to be a lattice with unit spacing, i.e. $\epsilon$ is rescaled to be unit spacing, Balaban in his series of papers showed that the sequence of finite dimensional integrals remains bounded, the bounds are independent of $\epsilon$. Thus, he showed ultraviolet stability, in the series of papers from \cite{volumeI} - \cite{1987-99_RG-05}.

In his work, it is important to note that he was giving meaning to the partition function $Z$. That is, given some action $S$, he defined the Gaussian integral \beq \int_{(\phi, A)}e^{S(\phi, A)}\ \mathcal{D}\phi\ \mathcal{D}A, \nonumber \eeq whereby $\phi$ is some scalar field, $A$ is a vector field and $\mathcal{D}\phi$ and $\mathcal{D}A$ are respectively Lebesgue measures. See \cite{volumeI, volumeII, volumeIII}, which are only applicable for dimensions $d=2, 3$. Using lattice approximation, he defined the partition function as a sequence of finite dimensional Gaussian integrals. And he showed that this sequence of integrals was bounded, with bounds independent of lattice spacing.

Among his series of papers from 1984 to 1989, only \cite{1985a}-\cite{1985e} and \cite{1987-99_RG-01}-\cite{1987-99_RG-05} are relevant to Yang-Mills gauge theory, using a compact semi-simple gauge group. The papers written from 1987 onwards were directed at dimension 4. These papers did not address the infrared divergence in the computation of the partition function. The action being considered in \cite{volumeI}-\cite{volumeIII} seemed to be directed at abelian gauge group, provided $\mu_0 = 0$ and there is no scalar field. Unfortunately, Balaban considered $\mu_0 > 0$ and a mass term was also introduced. Hence, he avoided the infrared divergence issue. Thus, these results are not useful for the Yang-Mills action using the abelian gauge group.  Finally, Balaban addressed the ultraviolet stability issue in \cite{volumeI, volumeII, volumeIII}.

\begin{rem}
In the case when $d = 2$, Brydges and his fellow authors in \cite{brydges01, brydges02, brydges03} showed that the Gaussian integral converges, without any cut-off, for $\mu_0^2 = 0$.
\end{rem}

\section{Renormalization flow}

For papers published from 1985 to 1989, Balaban would consider a bounded domain in the lattice $\epsilon \mathbb{Z}^d$, and introduced plaquette, with parallel translation between nodes. He would define gauge field configurations on the lattice and define a partition function, using a (finite) product Haar measure. Ultraviolet stability again is proved by showing that the partition function remains bounded, with bounds independent of $\epsilon$.

In contrast, we want to define a Gaussian probability measure, using the Abstract Wiener space formalism developed by Gross. The construction we are going to use comes from \cite{CSLim01}, which is based on the geometric quantization of the harmonic oscillator. However, we will use some of Balaban's renormalization ideas throughout this article in our construction.


Introduce a parameter $\kappa > 0$. In the works of Balaban, he would use $\epsilon$ to denote lattice spacing, whereby $\kappa$ is synonymous with $1/\epsilon$.
Consider the subspace $\mathcal{S}_\kappa(\mathbb{R}^4)$ inside $L^2(\bR^4)$, with the Gaussian function $\phi_\kappa$, \beq \sqrt{\phi_{\kappa}}(x) = \kappa^2 e^{-\kappa^2|x|^2/4}/(2\pi). \nonumber \eeq Any $f \in \mathcal{S}_\kappa(\mathbb{R}^4)$ can be written in the form \beq f(x) = p(x)\sqrt{\phi_{\kappa}}(x), \nonumber \eeq whereby $p$ is a polynomial in $x=(x^0, x^1, x^2, x^3)$.

The inner product $\langle \cdot, \cdot \rangle$ in question is given by \beq \langle f, g \rangle = \int_{\bR^4} f\cdot g\ d\lambda, \label{e.dr.4} \eeq $d\lambda$ is Lebesgue measure on $\mathbb{R}^4$. Let $\overline{\mathcal{S}}_\kappa(\bR^4)$ be the smallest Hilbert space containing $\mathcal{S}_\kappa(\mathbb{R}^4)$, using this inner product.

The Hermite polynomials $\{h_i\}_{i \geq 0}$ form an orthogonal set in $L^2(\mathbb{R}, d\mu)$ with respect to the Gaussian measure $d\mu(x_0) \equiv e^{-x_0^2/2}dx_0/\sqrt{2\pi}$. Let $\mathcal{P}_r $ be the set consisting of 4-tuples of the form $(i,j,k,l)$, each entry inside $\mathbb{N} \cup \{0\}$, such that the sum of its entries $i+j+k+l = r$. For $p_r = (i,j,k,l) \in \mathcal{P}_r$, define \beq H_{p_r}(x) := h_i(x^0) h_j(x^1)h_k(x^2)h_l(x^3) ,\nonumber \eeq a product of Hermite polynomials and $H_{p_r}^\kappa := H_{p_r}(\kappa\cdot)$.

We have the normalized Hermite polynomials $H_{p_r}/\sqrt{p_r!}$, $p_r! =i!j!k!l!$, with respect to the Gaussian measure $e^{-\sum_{i=0}^3|x^i|^2 /2}d\lambda/(2\pi)^2 $. Then  \beq  \bigcup_{r=0}^\infty  \left\{  H_{p_r}^\kappa \sqrt{\phi_\kappa}/\sqrt{p_r!}: p_r \in \mathcal{P}_r \right\}
\nonumber \eeq is an orthonormal basis for $\overline{\mathcal{S}}_\kappa(\bR^4)$.

\begin{defn}\label{d.shm.1}
Consider the real vector space spanned by $\{z^n:\ z \in \bC\}_{n=0}^\infty$, integrable with respect to the Gaussian measure, equipped with an inner product, given by
\begin{align}
\langle z^{r},& z^{r'}\rangle
= \frac{1}{\pi}\int_{\mathbb{C}}   z^{r} \cdot \overline{z^{r'}}e^{-|z|^2}
\ dx\ dp,\ z = x + \sqrt{-1}p.  \label{e.l.2}
\end{align}
Note that $\overline{z}$ means complex conjugate. Denote this (real) inner product space by $H^2(\bC)$, which consists of polynomials in $z$, with real coefficients.
\end{defn}

An orthonormal basis is given by \beq \left\{ \frac{z^n}{\sqrt{n!}}:\ n \geq 0 \right\}. \nonumber \eeq This inner product space $H^2(\bC)$ is the vector space for which we quantized the simple harmonic oscillator using geometric quantization. See \cite{MR1183739}. Complete it into a Hilbert space $\mathcal{H}^2(\bC)$.

Let $H^2(\bC^4) \equiv H^2(\bC)^{\otimes^4}$ and $H^2(\bC^3) \equiv H^2(\bC)^{\otimes^3}$, which consists of polynomials in $z=(z_0, z_1, z_2, z_3)$ and $(z_1, z_2, z_3)$ respectively. Write $d\lambda_4 = \frac{1}{\pi^4}e^{-\sum_{i=0}^3|z_i|^2} \prod_{i=0}^3 dx^i dp^i$ and define an inner product \beq \left\langle f, \hat{f} \right\rangle = \int_{\bC^4} f\cdot \overline{\hat{f}}\ d\lambda_4, \label{e.f.1} \eeq for $f$, $\hat{f} \in H^2(\bC^4)$. Its norm will be denoted by $\parallel \cdot \parallel$. Let $\mathcal{H}^2(\bC^4)$ and $\mathcal{H}^2(\bC^3)$ be the smallest Hilbert space containing $H^2(\bC^4)$ and $H^2(\bC^3)$ respectively. We will refer the former as the standard model of the quantum harmonic oscillator.

The Segal Bargmann transform $\Psi_\kappa$ actually maps Hermite polynomials in $\bR^4$ to complex holomorphic functions in $\bC^4$, equipped with the product Gaussian measure $\lambda_4$. See Equation (\ref{e.f.1}).

\begin{defn}(Segal Bargmann Transform)\\ \label{d.sbt}
Using the Segal Bargmann Transform $\Psi_\kappa: \overline{\mathcal{S}}_\kappa(\bR^4) \rightarrow \mathcal{H}^2(\bC^4)$, we map \beq \Psi_\kappa:  H_{p_r}(\kappa x^0, \kappa x^1, \kappa x^2, \kappa x^3)\sqrt{\phi_\kappa}/\sqrt{p_r!} \longmapsto \frac{z^{p_r}}{\sqrt{p_r}!} \equiv \frac{z_0^{i_0}z_1^{i_1}z_2^{i_2}z_3^{i_3}}{\sqrt{i_0!i_1!i_2!i_3!}}, \nonumber \eeq if $p_r = (i_0, i_2, i_2, i_3)$.
Clearly, it is an isometry.
\end{defn}

Suppose we identify $\overline{\mathcal{S}}_1(\mathbb{R}^4)$ with $\mathcal{H}^2(\bC^4)$, using $\Psi_1$. Then for each $\kappa > 0$, we can view the Segal Bargmann Transform as a map $\kappa^2 f(\kappa x) \sqrt{\phi_1}(\kappa x) \mapsto f(x) \sqrt{\phi_1}(x) $. We add an extra $\kappa^2$ so that the map is an isometry between Hilbert spaces, i.e. $\bigcup_{r \geq 0}\{[h_{p_r}(\kappa \cdot)/\sqrt{p_r!}] \sqrt{\phi_\kappa}:\ p_r \in \mathcal{P}_r \}$ remains an orthonormal basis in $\overline{\mathcal{S}}_\kappa(\mathbb{R}^4)$.

Because we may identify $\mathcal{S}_1(\mathbb{R}^4)$ with $H^2(\bC^4)$, we can view the Segal Bargmann Transform as a renormalization flow, whereby $\{\mathcal{S}_\kappa(\mathbb{R}^4):\ \kappa > 0\}$ is a continuous sequence of inner product spaces, which are transformed by $\Psi_\kappa: \mathcal{S}_\kappa(\mathbb{R}^4) \rightarrow H^2(\bC^4)$.

\begin{rem}\label{r.c.3}
\begin{enumerate}
\item This renormalization flow is analogous to the rescaling of the $\epsilon$-lattice to be a unit lattice, which is carried out repeatedly in Balaban's series of papers. See \cite{volumeII}, \cite{propagator1} and \cite{1985d}.
\item By identifying $\mathcal{S}_1(\mathbb{R}^4)$ with $H^2(\bC^4)$, we regard this as an analogue of the Fourier transform, which maps from position space to momentum space.
\item In \cite{propagator2}, Balaban went into the momentum representation, so the Gaussian integrals were defined on momentum space.
  \item Scaling is definitely evident in Balaban's work. See \cite{volumeI} and \cite{propagator1}. He used the scaling \beq \phi \longmapsto s^{-a}\phi(\cdot/s), \nonumber \eeq $a$ is some power. Here, we set $\kappa = 1/s$ and $a = 2$, i.e. $f\sqrt{\phi_1} \mapsto \kappa^2 f(\kappa \cdot)\phi_1(\kappa \cdot)$.
  \item In the physics textbook \cite{peskin1995introduction}, the transform was $x$ maps to $x/\kappa$.
\end{enumerate}
\end{rem}

\section{Gaussian Measure}\label{s.gm}

Using the inner product $\langle f, g \rangle = \int_{\bR^4} f(s)\cdot g(s)\ ds$, we want to make sense of a Gaussian integral on $L^2(\bR^4)$, of the form \beq \frac{1}{\tilde{Z}} \int_{f \in L^2(\bR^4)} e^{f(x_0)} e^{-\langle f, f \rangle/2} \mathcal{D}[f], \nonumber \eeq whereby $\mathcal{D}[f]$ is some non-existent Lebesgue measure on $L^2(\bR^4)$ and \beq \tilde{Z} = \int_{f \in L^2(\bR^4)}e^{-\langle f, f \rangle/2} \mathcal{D}[f] \nonumber \eeq is the partition function. Denote $\parallel f \parallel^2 = \langle f, f \rangle$.

Using this inner product, we will not work with the space $L^2(\bR^4)$. Instead, we will replace it with a smaller space $\overline{\mathcal{S}}_\kappa(\bR^4)$. With the above setup, we seek to define a path integral of the form \beq \frac{1}{Z}\int_{f \in \overline{\mathcal{S}}_\kappa(\mathbb{R}^4)} e^{f(x_0)} e^{-\parallel f\parallel^2/2} \mathcal{D}[f], \nonumber \eeq whereby $\frac{1}{Z}e^{-\parallel f\parallel^2/2} \mathcal{D}[f]$ is some form of Gaussian probability measure on an infinite dimensional Hilbert space $\overline{\mathcal{S}}_\kappa(\mathbb{R}^4)$.

\begin{rem}
\begin{enumerate}
  \item A finite gauge lattice with spacing $\epsilon$ is analogous to $\overline{\mathcal{S}}_\kappa(\mathbb{R}^4)$, whereby $\epsilon$ is synonymous to $1/\kappa$.
  \item Balaban seeked to obtain bounds on the partition function $Z$. In contrast, we never compute the partition function; instead we are constructing a probability measure.
\end{enumerate}
\end{rem}

Now, if we can write $f(x_0) = \langle f, \delta_{x_0} \rangle$ for some function $\delta_{x_0} \in \overline{\mathcal{S}}_\kappa(\mathbb{R}^4)$, then it would be easy to define our path integral. Unfortunately, such a function is given by a Dirac Delta function, which is a generalized function, and is not inside $\overline{\mathcal{S}}_\kappa(\mathbb{R}^4)$. But such a function exists in $\mathcal{H}^2(\bC^4)$.

\begin{defn}
Define a linear functional $\chi_w$ on $\mathcal{H}^2(\bC^4)$, $\chi_w(z) = e^{\bar{w}\cdot z}$, where \beq w = (w_0, w_1, w_2, w_3),\quad z = (z_0, z_1, z_2, z_3) \in \bC^4\quad {\rm and}\quad \bar{w} \cdot z \equiv \sum_{\alpha=0}^3 \overline{w_\alpha} z_\alpha. \nonumber \eeq Note that $\overline{w_\alpha}$ means the complex conjugate of $w_\alpha \in \bC$, and $|w|^2 = w \cdot \overline{w}$.

A direct calculation will show that for any $f \in \mathcal{H}^2(\bC^4)_\bC \equiv \mathcal{H}^2(\bC^4) \otimes_{\bR} \bC$, $\chi_w(z) \equiv e^{\bar{w}\cdot z}$ is the unique vector such that $\langle f, \chi_w \rangle = f(w)$ for any $w \in \bC^4$. Note that here, we extend $\langle \cdot, \cdot \rangle$ to be a sesquilinear complex inner product.
\end{defn}

We will term $\langle \cdot, \chi_w \rangle$ as an evaluation functional. Note that it is like the Dirac Delta distribution acting on continuous functions in $\bR^4$, except that here, $\chi_w$ is indeed a function and will play the role of a Dirac Delta function in $\mathcal{H}^2(\bC^4)$. By direct calculation, its norm is given by $\parallel \chi_w \parallel = e^{|w|^2/2}$.

Instead of making sense of the path integral over in $\overline{\mathcal{S}}_\kappa(\mathbb{R}^4)$ containing real Schwartz functions, we will now make sense over in $\mathcal{H}^2(\bC^4)$ containing holomorphic functions, via the Segal Bargmann transform. Hence, our path integral is now of the form \beq \frac{1}{\int_{f \in \mathcal{H}^2(\bC^4)} e^{-\parallel f \parallel^2/2}\ \mathcal{D}[f]}\int_{f \in \mathcal{H}^2(\bC^4)} e^{\langle f, \chi_{x_0}\rangle} e^{-\parallel f \parallel^2/2}\ \mathcal{D}[f], \nonumber \eeq whereby now $x_0 \in \bR^4 \hookrightarrow \bC^4$, i.e. we embed $\bR^4$ directly inside $\bC^4$.

\begin{rem}
\begin{enumerate}
  \item This is similar to how physicists define a path integral over in momentum space, via Fourier Transform.
  \item Using the Segal Bargmann transform as a renormalization flow is similar to how Balaban rescaled the finite gauge lattice of spacing $\epsilon$, to a finite gauge lattice of unit spacing.
  \item In the last Section \ref{s.cr}, we will compare this approach of defining a path integral, over the space containing holomporphic functions, with the approach of defining a path integral, over the space containing real Schwartz functions.
\end{enumerate}
\end{rem}

Let $H$ be a Hilbert space equipped with an inner product $\langle \cdot, \cdot \rangle$, and $P:H \rightarrow H$ be an orthogonal projection such that $PH$ is finite dimensional. Given any Borel subset $F  \subseteq PH$, define for $\theta > 0$,  \beq \mu_{\theta}\left( x \in P^{-1}(F) \right)= \left( \frac{\theta}{2\pi}\right)^{l/2}\int_{y \in F} e^{-\theta |y|^2 /2} dy , \label{e.n.2} \eeq where $l$ is the dimension of $P H$.

Unfortunately, $\mu_\theta$ does not extend to be a Gaussian measure on $H=\mathcal{H}^2(\bC^4)$, as explained in \cite{MR0461643}. We need to enlarge this Hilbert space in order to define a Gaussian measure on it.

This was done in \cite{CSLim01}, by defining a supremum norm $|\cdot|$ on $\mathcal{H}^2(\bC^4)$. For $f = \sum_{r \geq 0} \sum_{p_r \in \mathcal{P}_r} c_{p_r}\frac{z^{p_r}}{\sqrt{p_r!}} \in \mathcal{H}^2(\bC^4)$, we have \beq |f| := \sup_{z \in B(0, 1/2)}\sum_{r \geq 0}\sum_{p_r \in \mathcal{P}_r} |c_{p_r}||z^{p_r}|, \label{e.sup.1} \eeq $B(0,1/2)$ is the open ball with radius $1/2$, center $0$ in $\bC^4$.

By showing it is measurable according to Definition 4.4 in \cite{MR0461643}, we can complete the space $\mathcal{H}^2(\bC^4)$ into a Banach space containing holomorphic functions, denoted as $B(\bC^4)$.

\begin{rem}
Even though the construction in \cite{CSLim01}, was done over in $\mathcal{H}^2(\bC^3)$, the construction can be easily adapted for $\mathcal{H}^2(\bC^4)$.
\end{rem}

This Banach space $B(\bC^4)$ is then equipped with a Wiener measure $\mu_\theta$, i.e. $(B(\bC^4), \mu_\theta)$ is a probability space. The $\sigma$-algebra is the Borel field of $B(\bC^4)$.  Together with $\mathcal{H}^2(\bC^4)$ and an inclusion map $i: \mathcal{H}^2(\bC^4) \hookrightarrow B(\bC^4)$, they form a triple $(i, \mathcal{H}^2(\bC^4),B(\bC^4))$, an Abstract Wiener space, and thus we have the following containment of spaces \beq B(\bC^4)^\ast \subset \mathcal{H}^2(\bC^4) \subset B(\bC^4). \nonumber \eeq The Hilbert space is dense inside the Banach space and by abuse of notation, we identify the dual space $B(\bC^4)^\ast$, with a dense subspace contained inside the Hilbert space $\mathcal{H}^2(\bC^4)$. Indeed, this dual space is also the space of Gaussian random variables on $B(\bC^4)$. For details, refer to \cite{MR0461643}.

\begin{rem}
If one is able to map $B(\bC^4)$ back using the Segal Bargmann Transform, then we see that when we complete $\mathcal{S}_\kappa(\mathbb{R}^4)$ using a measurable norm, we get the space of generalized functions.
\end{rem}

It was shown in \cite{CSLim01} that $\chi_w$ lies in the complexified dual space $B(\bC^4)_\bC^\ast$ and the evaluation map $\phi \mapsto ( \phi, \chi_w )_\sharp:= \phi(w)$ is a complex random variable on $B(\bC^4)_\bC$. By definition of a Wiener measure $\mu_\theta$ on an Abstract Wiener space, we have
\begin{align*}
\int_{\phi \in B(\bC^4)} \phi(x)\phi(y)\ d\mu_\theta =&  \int_{\phi \in B(\bC^4)} ( \phi, \chi_x )_\sharp ( \phi, \chi_y )_\sharp\ d\mu_\theta  = \frac{1}{\theta}\langle \chi_x, \chi_y \rangle \equiv \frac{1}{\theta}e^{x \cdot y}.
\end{align*}

We computed the correlation function $e^{x\cdot y}$ between two space coordinates $x$ and $y$, both in $\bR^4$. Over the standard model of the quantum harmonic oscillator, the two distinct points are correlated. The physicists will call $e^{x\cdot y}$ as the propagator between $x$ and $y$. Note that the propagator $e^{x\cdot y}$ is finite, which is not a surprise, since we are doing it over the standard model of the quantum harmonic oscillator, which is synonymous to a finite lattice with unit spacing.

This is where we need to do a scaling. We need to identify $\bR^4$ inside $\bC^4$. Because of the renormalization flow $\Psi_\kappa$, we will map $x \in \bR^4$ into $\kappa x/2 \in \bR^4 \hookrightarrow \bC^4$. In other words, we embed $\bR^4$ inside $\bC^4$, but we scale by a factor $\kappa/2$. Thus, two distinct points $x, y \in \bR^4$, separated by a distance $|x-y|$, will now be mapped to $\kappa x/2$ and $\kappa y/2$ inside $\bC^4$ respectively, and their separation now will be $\kappa|x-y|/2$. If we relate $\kappa$ to the parameter $1/\epsilon$ in lattice gauge theory, then this embedding of $\bR^4$ inside $\bC^4$ with a scaling factor of $\kappa/2$, is analogous to the mapping of the lattice $\epsilon \mathbb{Z}^4$, each node separated by $\epsilon$ unit, to the standard lattice $\mathbb{Z}^4$, neighboring nodes separated by one unit.
Note that rescaling the lattice spacing will give us a renormalization transformation.

The correlation between distinct points $x$ and $y$ in $\bR^4$ or their propagator, will now be $\exp[\kappa^2 x\cdot y/4]$, which goes to infinity, unless $x\cdot y$ is non-positive, when $\kappa \rightarrow \infty$. This is nothing new, as the propagator in physics also has this problem. This is due to the divergent integrals that one obtains in Feynmann diagrams when computing the propagators. The reason for this divergence is because of the ultra-violet divergence issue.

\begin{rem}\label{r.p.1}
In \cite{volumeIII}, Balaban used Feynman diagrams to compute propagators.
\end{rem}

\begin{notation}\label{n.n.5}
Let $\psi_w = \psi(w) = e^{-|w|^2/2}/\sqrt{2\pi}$.
\end{notation}

In Quantum Field Theory, renormalization requires one to introduce a cutoff function in momentum space, so as to resolve this ultra-violet divergence issue, as was done in \cite{rivasseau01, peskin1995introduction}. Since we never use Fourier Transform to go into momentum space, we will instead normalize the evaluation functional $\chi_w$, by multiplying it by a factor $\psi_w = e^{-|w|^2/2}/\sqrt{2\pi}$, which is less than 1. There are two important reasons for doing this.

Now, at $w=0$, note that $\chi_0 = 1$ and its norm is $1$. But for $w \neq 0$, we have $\chi_w = e^{\bar{w}\cdot z}$ and its norm is $e^{|w|^2/2}$, so it is not normalized. Because $\bR^4$ is isotropic, i.e. there should be no preferred position in space, we should have that $\chi_w$ has the same norm, but this is not the case. Henceforth, we normalize $\chi_z$ and consider the normalized functional $\psi_z \chi_z$. Then, \beq \frac{1}{2\pi}\langle e^{-|w|^2/2}\chi_w, e^{-|z|^2/2} \chi_z \rangle = \frac{1}{2\pi}e^{-|z|^2/2} e^{-|w|^2/2} e^{\bar{w}z}. \nonumber \eeq

In particular, for real $x$ and $y$, the propagator becomes
\begin{align*}
\frac{1}{2\pi}\langle e^{-\kappa^2|x|^2/8}\chi_{\kappa x/2}, e^{-\kappa^2|y|^2/8} \chi_{\kappa y/2} \rangle =& \frac{1}{2\pi}e^{-\kappa^2|x|^2/8} e^{-\kappa^2|y|^2/8} e^{\kappa^2 x\cdot y/4} \\
=& \frac{1}{2\pi}e^{-\kappa^2|x-y|^2/8},
\end{align*}
after renormalization.

The second reason is this. The Dirac Delta function is not a true function. However, at the space point $x$, we can approximate it by a Gaussian function $(p_\kappa^x)^2$, which is denoted as $p_\kappa^x \equiv \kappa^2\exp[-\kappa^2|x-\cdot|^2/4]/(2\pi)$. Now, under the Segal Bargmann Transform, one can show that \beq \Psi_\kappa: \frac{1}{\sqrt{2\pi}}p_\kappa^x \longmapsto \chi_{\kappa x/2} \psi_{\kappa x/2}. \nonumber \eeq See \cite{CSLim01}. Since $\Psi_\kappa$ is an isometry, we have
\begin{align*}
\frac{1}{2\pi}\langle p_\kappa^x, p_\kappa^y \rangle =& \psi_{\kappa x/2} \psi_{\kappa y/2}\langle \chi_{\kappa x/2}, \chi_{\kappa y/2} \rangle \\
=& \psi_{\kappa x/2} \psi_{\kappa y/2}\exp[\kappa^2 x\cdot y/4]
= \frac{1}{2\pi}\exp[-\kappa^2|x-y|^2/8].
\end{align*}

The plan now is to make the separation between $\kappa x/2$ and $\kappa y/2$ go to infinity, by taking $\kappa$ go to infinity. This is analogous to taking the limit as the lattice spacing $\epsilon$ goes down to zero in lattice approximation, so that $\epsilon \mathbb{Z}^4$ will give us $\bR^4$. As $\kappa$ goes to infinity, the `renormalized' propagator goes to 0, if $x \neq y$. Thus any two distinct points become uncorrelated.

\section{Wiener space}\label{s.wm}

The Hermite polynomials satisfy the following property \beq h_{n+1}(x) = xh_n(x) - h_n'(x). \nonumber \eeq Therefore, \beq \frac{d}{dx}\left( h_n(x) e^{-x^2/4} \right) = \left[ xh_n(x) - h_{n+1}(x) - \frac{x}{2}h_n(x) \right] e^{-x^2/4}. \nonumber \eeq But we also have $h_n' = nh_{n-1}$, which means that $xh_n(x) = h_{n+1}(x) + nh_{n-1}(x)$. Thus,
\begin{align}
\frac{d}{dx}\left( h_n(x) e^{-x^2/4} \right) =& \left[ \frac{1}{2}h_{n+1}(x) + \frac{n}{2}h_{n-1}(x) - h_{n+1}(x) \right] e^{-x^2/4} \nonumber \\
=& \left[\frac{n}{2}h_{n-1}(x)-\frac{1}{2}h_{n+1}(x) \right]e^{-x^2/4}. \nonumber
\end{align}

\begin{defn}
Define a linear operator $\mathfrak{d}_a: H^2(\bC) \rightarrow H^2(\bC)$ by \beq \mathfrak{d}_a z_a^n= \frac{n}{2}z_a^{n-1} - \frac{1}{2}z_a^{n+1}, \label{e.x.3} \eeq for $a = 0, 1, 2, 3$. Extend it to be an operator on $H^2(\bC^4)$, by \beq \mathfrak{d}_a\left[z_a^p \prod_{b \neq a \atop b=0,\cdots 3} z_b^{q_b} \right] = \left[\frac{p}{2}z_a^{p-1} - \frac{1}{2}z_a^{p+1} \right]\cdot \prod_{b \neq a \atop b=0,\cdots 3} z_b^{q_b},\ \ p, q_b \in \mathbb{N} \cup \{0\}. \nonumber \eeq We denote the range of $\fd_a$ by $\fd_a H^2(\bC^4)$.
\end{defn}

\begin{prop}\label{p.p.1}
For any $f \in H^2(\bC)$, we have for some positive $C$, \beq \langle f, f \rangle \leq C \langle \fd_a f, \fd_a f \rangle. \label{e.a.2} \eeq
\end{prop}

\begin{proof}
By definition of $\fd_a$, we note that for any $n \geq 0$, \beq \fd_a \frac{z_a^n}{\sqrt{n!}} = \frac{\sqrt{n}}{2}\frac{z_a^{n-1}}{\sqrt{(n-1)!}} - \frac{\sqrt{n+1}}{2}\frac{z_a^{n+1}}{\sqrt{(n+1)!}}. \label{e.z.1} \eeq Using the fact that \beq \left\langle z_a^n/\sqrt{n!}, z_a^m/\sqrt{m!} \right\rangle =
\left\{
  \begin{array}{ll}
    1, & \hbox{$m=n$;} \\
    0, & \hbox{$m\neq n$,}
  \end{array}
\right.
 \nonumber \eeq we have
\begin{align*}
\left\langle \fd_a z_a^n/\sqrt{n!}, \fd_a z_a^n/\sqrt{n!} \right\rangle =& \frac{n}{4} + \frac{n+1}{4} = \frac{2n+1}{4}, \\
\left\langle \fd_a z_a^n/\sqrt{n!}, \fd_a z_a^{n+2}/\sqrt{(n+2)!} \right\rangle =& -\frac{\sqrt{n+1}\sqrt{n+2}}{4}.
\end{align*}
Let $H_e = {\rm span\ of}\ \{z_a^{m}:\ m\ {\rm even}\}$ and $H_o = {\rm span\ of}\ \{z_a^{m}:\ m\ {\rm odd}\}$. Note that $\langle \fd_a f, \fd_a g \rangle = 0$, if $f \in H_e$ and $g \in H_o$. So, it suffices to show that the result holds on $H_e$ and $H_o$ separately. But note that $\langle \fd_a z_a^m, \fd_a z_a^{m+2}\rangle \neq 0$.

Let $f_e(z_a) = \sum_{r=0}^\infty c_{2r} z_a^{2r}/\sqrt{(2r)!} \in H^2(\bC)$, whereby $c_{2r} \in \bR$. Note that \beq \parallel f_e \parallel^2 = \sum_{r=0}^\infty c_{2r}^2 = \sum_{r=0}^\infty c_{4r}^2 + \sum_{r=0}^\infty c_{2(2r+1)}^2, \nonumber \eeq so without loss of generality, we assume that $\sum_{r=0}^\infty c_{4r}^2 \geq \parallel f_e \parallel^2/2$.

Let $\theta$ be the angle between $\fd_a[z_a^{4r}]/\sqrt{(4r)!}$ and $\fd_a[z_a^{4r+2}]/\sqrt{(4r+2)!}$. Then, \beq \cos \theta = -\frac{\sqrt{4r+1}\sqrt{4r+2}}{\sqrt{8r+1}\sqrt{8r+5}}. \nonumber \eeq Note that for any $r=0,1,2,\cdots$, we have \beq \frac{\sqrt{4r+1}\sqrt{4r+2}}{\sqrt{8r+1}\sqrt{8r+5}} \leq \frac{\sqrt2}{\sqrt5}. \nonumber \eeq Furthermore, \beq \lim_{r \rightarrow \infty}\frac{\sqrt{4r+1}\sqrt{4r+2}}{\sqrt{8r+1}\sqrt{8r+5}} = \frac{1}{2}. \nonumber \eeq

Let $P_+[z_a^{4r}/\sqrt{(4r)!}]$ be the projection of $\fd_a[z_a^{4r}]/\sqrt{(4r)!}$ on the span of $\{\fd_a[z_a^{4r+2}]\}$. Then, \beq \parallel P_+[z_a^{4r}/\sqrt{(4r)!}] \parallel \leq \sqrt{2/5}\parallel \fd_a[z_a^{4r}]/\sqrt{(4r)!} \parallel.  \nonumber \eeq

Similarly, if we let $P_-[z_a^{4r}/\sqrt{(4r)!}]$ be the projection of $\fd_a[z_a^{4r}]/\sqrt{(4r)!}$ on the span of $\{\fd_a[z_a^{4r-2}]\}$, then \beq \parallel P_-[z_a^{4r}/\sqrt{(4r)!}] \parallel \leq \sqrt{2/5}\parallel\fd_a[ z_a^{4r}]/\sqrt{(4r)!} \parallel. \nonumber \eeq

Let $H_+$ be the span of $\{\fd_a z_a^{4r+2}:\ r=0, 1, 2, \cdots \}$ and $P_\perp[z_a^{4r}/\sqrt{(4r)!}]$ be the projection of $\fd_a[z_a^{4r}]/\sqrt{(4r)!}$ on the orthogonal complement of $H_+$. Since $P_+[z_a^{4r}/\sqrt{(4r)!}]$ and $P_-[z_a^{4r}/\sqrt{(4r)!}]$ are orthogonal, we see that \beq \parallel P_\perp[z_a^{4r}/\sqrt{(4r)!}] \parallel \geq \sqrt{1 - (\sqrt{2/5})^2- (\sqrt{2/5})^2}\parallel \fd_a[z_a^{4r}]/\sqrt{(4r)!}\parallel = \frac{1}{\sqrt{5}}\sqrt{\frac{8r+1}{4}}. \nonumber \eeq

Because $\{ \fd_a[z_a^{4r}]:\ r=0, 1, 2, \cdots\}$ is an orthogonal set, therefore
\begin{align*}
 \parallel \fd_af_e \parallel^2 \geq& \left\Vert P_\perp\left[\sum_{r=0}^\infty c_{4r} z_a^{4r}/\sqrt{(4r)!} \right] \right\Vert^2 \geq \frac{1}{5} \sum_{r=0}^\infty c_{4r}^2 \frac{8r+1}{4} \\
\geq& \frac{1}{20} \sum_{r=0}^\infty c_{4r}^2
\geq \frac{1}{20}\frac{\parallel f_e \parallel^2}{2}.
\end{align*}

Hence, \beq \parallel f_e \parallel \leq \sqrt{40}\parallel \fd_a f_e \parallel. \nonumber \eeq Let $f_o(z_a) = \sum_{r=0}^\infty c_{2r+1} z_a^{2r+1}/\sqrt{(2r+1)!} \in H^2(\bC)$, whereby $c_{2r+1} \in \bR$. A similar argument will also show that \beq \parallel f_o \parallel \leq \sqrt{40}\parallel \fd_a f_o \parallel. \nonumber \eeq Let $f(z_a) = \sum_{r=0}^\infty c_{r} z_a^{r}/\sqrt{r!} \in H^2(\bC)$, which we will write as $f = f_e + f_o$, $f_e \in H_e$ and $f_o \in H_o$. Since $H_e$ and $H_o$ are orthogonal, we thus have
\begin{align*}
\parallel f \parallel^2 &= \parallel f_e \parallel^2 + \parallel f_o\parallel^2 \leq \left(\sqrt{40}\right)^2\parallel \fd_a f_e \parallel^2 + \left(\sqrt{40}\right)^2\parallel \fd_a f_o \parallel^2 \\
&= \left(\sqrt{40}\right)^2\parallel \fd_a f \parallel^2.
\end{align*}

\end{proof}

Let us comment on Equation (\ref{e.a.2}). In general, this inequality does not hold for any given function $f$ and its derivative, both in $L^2(\bR)$. When one takes Fourier transform as in \cite{rivasseau01}, one sees that we have the term $\langle p_0^2 \hat{f}, \hat{f} \rangle$ in the Yang-Mills action. This is where the problem comes in, as this means that the propagator computed in Feynmann diagrams, will suffer from infrared divergence. But if we restrict the connection to be inside the space $\mathcal{S}_\kappa(\bR^4)\otimes \Lambda^1(\bR^3)$, Proposition \ref{p.p.1} will imply that the (abelian) Yang-Mills action given in Equation (\ref{e.ym.3}) is non-degenerate. Refer back to Remark \ref{r.ym.1}.

We wish to qualify that Equation (\ref{e.a.2}) does not imply that there is a mass gap, a millennium problem described in \cite{jaffe2006quantum}. In \cite{Araki:1962zhd}, the authors showed that the force follows an inverse square law decay, if there is no mass gap. As abelian Yang-Mills theory describe the electromagnetic field theory, hence there cannot be a mass gap.

\begin{notation}\label{n.d.1}
Recall we defined a measurable norm $|\cdot|$ on $B(\bC^4)$ in Equation (\ref{e.sup.1}). Consider the subspace inside $B(\bC^4)$ such that \beq \{ x \in B(\bC^4):\ |\fd_0 x| < \infty\}, \nonumber \eeq and close this subspace denoted by $B_0(\bC^4)$, using the norm $\parallel x \parallel := |x| + |\fd_0 x|$. Note that $H^2(\bC^4)$ is dense inside it.
\end{notation}

From the definition of the norm $\parallel \cdot \parallel$, we see that $\fd_0: B_0(\bC^4) \rightarrow B(\bC^4)$ is a bounded linear operator.

\begin{defn}
Write $f = \sum_{r \geq 0} \sum_{p_r \in \mathcal{P}_r} c_{p_r}\frac{z^{p_r}}{\sqrt{p_r!}} \in \mathcal{H}^2(\bC^4)$. 

For $p_r = (i,j,k,l)$, write $p_r^\pm = (i\pm1, j, k, l)$. Introduce a norm $| \cdot |_0$, given by \beq |f|_0 := \sup_{z \in B(0, 1/2)}\sum_{r \geq 0}\sum_{(i,j,k,l)\equiv p_r \in \mathcal{P}_r} \left[\frac{\sqrt i}{2}|z^{p_r^-}|+\frac{\sqrt{i+1}}{2}|z^{p_r^+}|\right]|c_{p_r}|, \label{e.sup.2} \eeq whereby
$B(0,1/2)$ is the open ball with radius $1/2$, center $0$ in $\bC^4$.
\end{defn}

\begin{prop}\label{p.p.2}
The norm $\parallel \cdot \parallel$ on $B_0(\bC^4)$ is a measurable norm.
\end{prop}

\begin{proof}
Recall we defined a norm $|\cdot|$ on $B(\bC^4)$. Consider the span  \beq S(N_0):= {\rm span}\ \left\{z_0^i z_1^j z_2^k z_3^l:\ i+j + k+l \leq N_0 \right\} \subset \mathcal{H}^2(\bC^4). \nonumber \eeq

Let $\epsilon > 0$. 
The proof of Proposition 2.6 in \cite{CSLim01} can be used to show that both $|\cdot|$ and $|\cdot|_0$ are measurable norms, by choosing a $N_0 \in \mathbb{N}$ large such that for any finite dimensional orthogonal projection $P: \mathcal{H}^2(\bC^4) \rightarrow \mathcal{H}^2(\bC^4)$, with the range of $P \mathcal{H}^2(\bC^4)$ being finite dimensional and orthogonal to $S(N_0+1)$, we have that \beq \mu_\theta\left\{x \in \mathcal{H}^2(\bC^4):\ |Px| > \epsilon\right\} < \epsilon,\quad {\rm and} \quad \mu_\theta\left\{x \in \mathcal{H}^2(\bC^4):\ |Px|_0 > \epsilon\right\} < \epsilon. \nonumber \eeq
Refer to Equation (\ref{e.n.2}).

Write for $\alpha = 0, 1, 2, 3$, \beq \mathcal{P}_r^\alpha = \left\{(i,j,k,l):\ i = \alpha + 4r, \ r, j,k,l \in \mathbb{N} \cup \{0\} \right\}. \nonumber \eeq For $x \in B_0(\bC^4)$, write $x_e^-$ and $x_e^+$ to be its orthogonal projections $P_e^{-}$ and $P_e^{+}$,
onto the span of $\{z_0^{4r}z_1^jz_2^kz_3^l: r,j,k,l \in \mathbb{N} \cup \{0\}\}$ and $\{z_0^{4r+2}z_1^jz_2^kz_3^l: r,j,k,l \in \mathbb{N} \cup \{0\}\}$ respectively. Similarly, write $x_o^-$ and $x_o^+$ to be its orthogonal projections $P_o^{-}$ and $P_o^{+}$, onto the span of $\{z_0^{4r+1}z_1^jz_2^kz_3^l: r,j,k,l \in \mathbb{N} \cup \{0\}\}$ and $\{z_0^{4r+3}z_1^jz_2^kz_3^l: r,j,k,l \in \mathbb{N} \cup \{0\}\}$ respectively.

Refer to Equation (\ref{e.sup.2}). If $x = \sum_{r \geq 0} \sum_{p_r \in \mathcal{P}_r} c_{p_r}\frac{z^{p_r}}{\sqrt{p_r!}}$, then we have from Equation (\ref{e.z.1}),
\begin{align*}
|\fd_0 x_e^-| = \sup_{z \in B(0, 1/2)}\sum_{r \geq 0}\sum_{p_r \in \mathcal{P}_r^0} \left[\frac{\sqrt i}{2}|z^{p_r^-}|+\frac{\sqrt{i+1}}{2}|z^{p_r^+}|\right]|c_{p_r}| \equiv |P_e^-x|_0, \\
|\fd_0 x_o^-| = \sup_{z \in B(0, 1/2)}\sum_{r \geq 0}\sum_{p_r \in \mathcal{P}_r^1} \left[\frac{\sqrt i}{2}|z^{p_r^-}|+\frac{\sqrt{i+1}}{2}|z^{p_r^+}|\right]|c_{p_r}| \equiv |P_o^-x|_0, \\
|\fd_0 x_e^+| = \sup_{z \in B(0, 1/2)}\sum_{r \geq 0}\sum_{p_r \in \mathcal{P}_r^2} \left[\frac{\sqrt i}{2}|z^{p_r^-}|+\frac{\sqrt{i+1}}{2}|z^{p_r^+}|\right]|c_{p_r}| \equiv |P_e^+ x|_0, \\
|\fd_0 x_o^+| = \sup_{z \in B(0, 1/2)}\sum_{r \geq 0}\sum_{p_r \in \mathcal{P}_r^3} \left[\frac{\sqrt i}{2}|z^{p_r^-}|+\frac{\sqrt{i+1}}{2}|z^{p_r^+}|\right]|c_{p_r}| \equiv |P_o^+ x|_0.
\end{align*}

Hence, for any finite dimensional orthogonal projection $P$, whose finite dimensional range is orthogonal to $S(N_0+1)$, we see that
\begin{align*}
\mu_\theta\left\{ \parallel Px \parallel > \epsilon \right\} =& \mu_\theta\left\{ | Px | + |\fd_0 Px| > \epsilon \right\} \leq \mu_\theta\left\{ | Px | > \epsilon/2 \right\} + \mu_\theta\left\{  |\fd_0 Px| > \epsilon/2 \right\} \\
\leq& \mu_\theta\left\{ | Px | > \epsilon/2 \right\} + \mu_\theta\left\{  |\fd_0 P_o^+Px| > \epsilon/8 \right\}+ \mu_\theta\left\{  |\fd_0 P_o^-Px| > \epsilon/8 \right\} \\
&+ \mu_\theta\left\{  |\fd_0 P_e^+Px| > \epsilon/8 \right\} + \mu_\theta\left\{  |\fd_0 P_e^-Px| > \epsilon/8 \right\} \\
=& \mu_\theta\left\{ | Px | > \epsilon/2 \right\} + \mu_\theta\left\{  |P_o^+Px|_0 > \epsilon/8 \right\}+ \mu_\theta\left\{  |P_o^-Px|_0 > \epsilon/8 \right\} \\
&+ \mu_\theta\left\{  |P_e^+Px|_0 > \epsilon/8 \right\} + \mu_\theta\left\{  |P_e^-Px|_0 > \epsilon/8 \right\} \\
\leq& \mu_\theta\left\{ | Px | > \epsilon/2 \right\} + 4\mu_\theta\left\{  |Px|_0 > \epsilon/8 \right\} \\
\leq& \frac{\epsilon}{2} + 4\frac{\epsilon}{8} = \epsilon.
\end{align*}
\end{proof}

\begin{defn}\label{d.wm.1}
Using the Abstract Wiener space formalism, we can define a Wiener measure on $(B_0(\bC^4), \parallel \cdot \parallel)$, which we will denote by $\mu_\theta$, with variance $1/\theta$. Note that $B_0(\bC^4)$ is dense inside $B(\bC^4)$.
\end{defn}

\begin{prop}\label{p.p.3}
We have that $1 \in \fd_0 B_0(\bC^4)$.
\end{prop}

\begin{proof}
We need to show the existence of a $g \in B_0(\bC^4)$ such that $\fd_0 g = 1$ and $\parallel g \parallel < \infty$. Now write $g = \sum_{n=0}^\infty c_{2n+1}z_0^{2n+1}$. By definition of $\fd_0$, we have that \beq \fd_0 g = \frac{1}{2}c_1 + \sum_{n=1}^\infty \frac{1}{2}\left[(2n+1)c_{2n+1} - c_{2n-1} \right]z_0^{2n} = 1. \nonumber \eeq Thus, we must have that $c_1 = 2$ and the recursive relation $c_{2n+1} = \frac{1}{2n+1}c_{2n-1}$ for $n \geq 1$. Using an induction argument, we see that \beq c_{2n+1} = \frac{1}{2n+1}\frac{1}{2n-1}\cdots \frac{2}{3},\ n \geq 1. \nonumber \eeq It is clear that $\sqrt{(2n+1)!}c_{2n+1} \leq 2$ for all $n \in \mathbb{N}$. Hence, \beq \parallel g \parallel = |g| + |\fd_0 g| \leq 1 + 2\sum_{n=0}^\infty \frac{1}{2^{2n+1}} < \infty. \nonumber \eeq Hence the proof.
\end{proof}

\begin{cor}\label{c.w.5}
We have that $H^2(\bC^4) \subset \fd_0 B_0(\bC^4)$.
\end{cor}

\begin{proof}
From the definition of $\fd_0$, it is clear that \beq {\rm span}\ \left\{  z_0^{2i+1} z_1^j z_2^k z_3^l:\ i,j,k,l \in \mathbb{N} \cup \{0\} \right\} = {\rm span}\ \left\{\fd_0 z_0^{2i} z_1^j z_2^k z_3^l:\ i,j,k,l \in \mathbb{N}\cup \{0\} \right\} , \nonumber \eeq which is a subspace in $\fd_0 B_0(\bC^4)$. And we also have that \beq {\rm span}\ \{1\} + {\rm span}\ \left\{ \fd_0 z_0^{2n-1}:\ n \in \mathbb{N} \right\} = {\rm span}\ \left\{ z_0^{2n}:\ n \in \mathbb{N}\cup\{0\} \right\}. \nonumber \eeq Hence, we have that \beq {\rm span}\ \left\{ z_0^{i} z_1^j z_2^k z_3^l:\ i,j,k,l \in \mathbb{N}\cup \{0\} \right\} \subset \fd_0 B_0(\bC^4), \nonumber \eeq thus $H^2(\bC^4) \subset \fd_0 B_0(\bC^4)$.
\end{proof}

\begin{rem}\label{r.d.1}
From this, we can see that $H^2(\bC^4)$ is a dense subspace inside $\fd_0 B_0(\bC^4)$. In fact, the dense inner product space we need to consider is $H^2(\bC^3) + \fd_0 H^2(\bC^4)$, which is dense inside $\fd_0 B_0(\bC^4)$. Recall $H^2(\bC^3)$ consists of polynomials of $(z_1, z_2, z_3)$.

Complete the inner product space  $\left(H^2(\bC^3) + \fd_0 H^2(\bC^4), \langle \cdot, \cdot \rangle\right)$ into a Hilbert space, denoted as $\mathcal{H}^2(\bC^3) + \overline{\fd_0 H^2(\bC^4)} \subseteq \mathcal{H}^2(\bC^4)$. Note that for $y \in \fd_0 H^2(\bC^4)$, $\langle y, y \rangle \equiv \langle \fd_0 x, \fd_0 x \rangle$ for some unique $x \in H^2(\bC^4)$, the inner product given in Equation (\ref{e.f.1}).
\end{rem}

\begin{prop}
The map $\fd_0: B_0(\bC^4) \rightarrow B(\bC^4)$ is a bounded and injective map.
\end{prop}

\begin{proof}
We already stated that it is bounded. Suppose $\fd_0 x = 0$ for $x \in B_0(\bC^4)$. Since $H^2(\bC^4)$ is dense inside $B_0(\bC^4)$, we can find a sequence $\{x_n \in H^2(\bC^4)\}_{n=0}^\infty$ such that $\parallel x_n-x \parallel \rightarrow 0$, as $n \rightarrow \infty$. Since $\fd_0$ is continuous on $B_0(\bC^4)$, we see that $\fd_0 x_n \rightarrow 0$ as $n \rightarrow \infty$. But Equation (\ref{e.a.2}) will then imply that $x_n \rightarrow 0$ as $n \rightarrow \infty$. Hence $x = 0$.
\end{proof}

\begin{defn}\label{d.wm.2}
Because $\fd_0$ is an injective map, define a norm on $\fd_0 B_0(\bC^4) \subset B(\bC^4)$ by \beq \parallel \fd_0 x \parallel := | \fd_0 x | + |x|,\ x \in B_0(\bC^4). \nonumber \eeq Hence, $(\fd_0 B_0(\bC^4), \parallel \cdot \parallel)$ is also a Banach space containing $\mathcal{H}^2(\bC^3) + \overline{\fd_0 H^2(\bC^4)}$. Note that $\fd_0$ is now an isometry map.
\end{defn}

\begin{prop}\label{p.w.1}
We have that $\parallel \cdot \parallel$ on $\fd_0 B_0(\bC^4)$ is also measurable, hence we can equip $(\fd_0 B_0(\bC^4), \parallel \cdot \parallel)$ with a Wiener measure. By abuse of notation, we will still use the same symbol $\mu_\theta$ to denote this Wiener measure, with variance $1/\theta$.
\end{prop}

\begin{proof}
Consider the span  \beq S(N_0):= {\rm span}\ \left\{z_0^i z_1^j z_2^k z_3^l:\ i+j + k+l \leq N_0 \right\} \subset \mathcal{H}^2(\bC^4). \nonumber \eeq In Proposition \ref{p.p.3}, we showed that there is a $g \in B_0(\bC^4)$ such that $\fd_0 g = 1$. Now for any $f = \sum_{r=0}^\infty \sum_{p_r \in \mathcal{P}_r}c_{p_r}\frac{z^{p_r}}{\sqrt{p_r!}}$ such that $\langle \fd_0 f, \fd_0 f \rangle = 1$,  \beq \langle f, f \rangle \leq C\langle \fd_0 f, \fd_0 f\rangle = C \nonumber \eeq from Equation (\ref{e.a.2}) will imply that $|c_{p_r}|^2 \leq C$. Let $\tilde{H} = \overline{(H^2(\bC^4), \langle \fd_0\cdot, \fd_0\cdot\rangle)}$ be the Hilbert space containing $H^2(\bC^4)$, using the inner product $(f,g) \mapsto \langle \fd_0 f, \fd_0 g \rangle$, for $f, g \in H^2(\bC^4)$.

Let $\epsilon > 0$. We can modify the proofs of Proposition 2.6 in \cite{CSLim01} and  Proposition \ref{p.p.2}, to show that by choosing $N_0$ large enough, for any finite dimensional orthogonal projection $\bar{P}: \tilde{H} \rightarrow \tilde{H}$ and $\bar{P}\tilde{H}$ is orthogonal to $S(N_0+1)$, we have using Equation (\ref{e.n.2}), \beq \nu_\theta\left\{x \in \tilde{H}:\ \parallel \bar{P}x \parallel > \epsilon \right\} = \nu_\theta\left\{x \in \tilde{H}:\ | \bar{P}x | + |\fd_0 \bar{P}x| > \epsilon \right\} < \epsilon. \nonumber \eeq Similarly, for any finite dimensional orthogonal projection $P: \mathcal{H}^2(\bC^3) \rightarrow \mathcal{H}^2(\bC^3)$ and $P\mathcal{H}^2(\bC^3)$ orthogonal to $S(N_0)$, we have using Equation (\ref{e.n.2}), \beq \mu_\theta\left\{x \in (\mathcal{H}^2(\bC^3), \langle \cdot, \cdot \rangle):\  | Px |> \epsilon \right\} < \epsilon. \nonumber \eeq

Write $H = \mathcal{H}^2(\bC^3) + \overline{\fd_0 H^2(\bC^4)}$, and $P: H \rightarrow H$ be any finite dimensional orthogonal projection, with the range $PH$ being orthogonal to $S(N_0)$. We see that for any $y = y_1 + y_2 \in P\mathcal{H}^2(\bC^3) + P\overline{\fd_0 H^2(\bC^4)}$, we can find a finite dimensional orthogonal projection $\bar{P}: \tilde{H} \rightarrow \tilde{H}$, such that $y_2 = \fd_0 x_2$ for an unique $x_2 \in \bar{P}\tilde{H}$. It is not difficult to see that \beq \nu_\theta\left\{x \in \tilde{H}:\ | \bar{P}x | + |\fd_0 \bar{P}x| > \epsilon \right\} = \mu_\theta\left\{\fd_0  x_2 \in \mathcal{H}^2(\bC^4):\ |\bar{P}x_2| + |\fd_0\bar{P}x_2| > \epsilon\right\}.\nonumber \eeq
Since $Py_2 = \fd_0 \bar{P}x_2 \perp S(N_0)$, therefore $\bar{P}x_2 \perp S(N_0 + 1)$ is unique, and
\begin{align*}
\mu_\theta&\left\{y \in H:\ \parallel Py \parallel > \epsilon\right\} \\
\leq& \mu_\theta\left\{\fd_0 x_2 \in \mathcal{H}^2(\bC^4):\ |\bar{P}x_2| + |\fd_0\bar{P}x_2| > \epsilon/2\right\} \\
&+ \mu_\theta\left\{x_1 \in \mathcal{H}^2(\bC^3):\ |g \otimes Px_1| + |Px_1| > \epsilon/2\right\} \\
\leq& \nu_\theta\left\{ x_2 \in \tilde{H}:\ |\bar{P}x_2| + |\fd_0\bar{P}x_2| > \epsilon/2\right\} \\
&+ \mu_\theta\left\{x_1 \in \mathcal{H}^2(\bC^3):\ |g| |Px_1| + |Px_1| > \epsilon/2\right\} \\
\leq& \frac{\epsilon}{2} + \mu_\theta\left\{x_1 \in \mathcal{H}^2(\bC^3):\ (1+|g|) |Px_1| > \epsilon/2\right\} \leq \frac{\epsilon}{2} + \frac{\epsilon}{2(1 + |g|)} < \epsilon.
\end{align*}
\end{proof}

\section{Segal Bargmann Transform}\label{s.bgt}

\begin{notation}
Recall we use the standard metric on $T\bR^4$ and the volume form on $\bR^4$ is given by $d\omega = dx^0 \wedge dx^1 \wedge dx^2 \wedge dx^3$. Using the Hodge star operator and the above volume form, we will define an inner product on $\mathcal{S}_\kappa(\bR^4)\otimes \Lambda^2(\bR^4)$ by \beq \left\langle \sum_{0\leq a< b \leq3} g_{ab} \otimes dx^a\wedge dx^b, \sum_{0\leq a< b \leq3} \hat{g}_{ab} \otimes dx^a\wedge dx^b \right\rangle = \sum_{0\leq a< b \leq3} \left\langle g_{ab}, \hat{g}_{ab} \right\rangle, \nonumber \eeq $g_{ab}$ and $\hat{g}_{ab} \in \mathcal{S}_\kappa(\bR^4)$.

We will further define an inner product on $H^2(\bC^4)\otimes \Lambda^2(\bR^4)$ by \beq \left\langle \sum_{0\leq a< b \leq3} f_{ab} \otimes dx^a\wedge dx^b, \sum_{0\leq a< b \leq3} \hat{f}_{ab} \otimes dx^a\wedge dx^b \right\rangle = \sum_{0\leq a< b \leq3} \left\langle f_{ab}, \hat{f}_{ab} \right\rangle. \nonumber \eeq Here, $f_{ab}$ and $\hat{f}_{ab}$ are in $H^2(\bC^4)$. We will continue to use $\langle\cdot, \cdot \rangle $ and $\parallel\cdot \parallel$ to denote the inner product and norm on $H^2(\bC^4)\otimes \Lambda^2(\bR^4)$ respectively.

We can extend the Segal Bargmann Transform to tensor and direct products and by abuse of notation, use the same symbol. That is, $\Psi_\kappa:  \mathcal{S}_\kappa(\bR^4) \otimes \Lambda^1(\bR^3) \rightarrow H^2(\bC^4) \otimes \Lambda^1(\bR^3)$ by \beq \Psi_\kappa \left[ \sum_{1\leq i\leq 3} A_i \otimes dx^i  \right] = \sum_{1\leq i \leq 3} \Psi_\kappa [A_i] \otimes dx^i. \nonumber \eeq

And $\Psi_\kappa:  \mathcal{S}_\kappa(\bR^4) \otimes \Lambda^2(\bR^4) \rightarrow H^2(\bC^4) \otimes \Lambda^2(\bR^4)$ by \beq \Psi_\kappa \left[ \sum_{0\leq a< b\leq 3} f_{a,b} \otimes dx^a \wedge dx^b \right] = \sum_{0\leq a<b \leq 3} \Psi_\kappa [f_{a,b}] \otimes dx^a \wedge dx^b. \nonumber \eeq
\end{notation}

\begin{defn}\label{d.b.1}
Let the span of $\{dx^0 \wedge dx^i, 1\leq i \leq 3\}$ and the span of $\{dx^i \wedge dx^j, 1\leq i < j\leq 3\}$ be denoted by $\ast \Lambda^2(\bR^3)$ and $\Lambda^2(\bR^3)$ respectively. Define
\begin{align*}
\mathbb{H}&:= \left\{ [\fd_0 H^2(\bC^4)] \otimes \left[\ast \Lambda^2(\bR^3) \right]\right\} \oplus \left\{ H^2(\bC^4) \otimes \Lambda^2(\bR^3)\right\}
\subset H^2(\bC^4) \otimes \Lambda^2(\bR^4).
\end{align*}
\end{defn}

In Section \ref{s.wm}, we explained how one can rigorously construct a Gaussian measure over a Banach space $B_0(\bC^4)$ and its range $\fd_0 B_0(\bC^4)$, each containing $H^2(\bC^4)$ and $\mathcal{H}^2(\bC^3) + \fd_0 H^2(\bC^4)$ respectively. We can enlarge the inner product space $\mathbb{H}$  into a Banach space \beq \dB:= \left\{[\fd_0 B_0(\bC^4)] \otimes [\ast\Lambda^2(\bR^3)] \right\}\oplus \left\{ B_0(\bC^4) \otimes \Lambda^2(\bR^3)\right\} \subset B(\bC^4) \otimes \Lambda^2(\bR^4), \nonumber \eeq using a measurable norm $\parallel \cdot \parallel$, defined as
\beq \parallel \tilde{B} \parallel = \sum_{i=1}^3 \parallel A_i \parallel + \sum_{1 \leq i < j \leq 3}\parallel B_{ij} \parallel, \label{e.n.5} \eeq using the norms defined in Notation \ref{n.d.1} and Definition \ref{d.wm.2}, for $\tilde{B} = \sum_{i=1}^3 A_i \otimes dx^0 \wedge dx^i + \sum_{1 \leq i < j \leq 3}B_{ij}\otimes dx^i \wedge dx^j \in \dB$. This Banach space contains a dense inner product space $H^2(\bC^4) \otimes \Lambda^2(\bR^4)$. Refer to Remark \ref{r.d.1}.

\begin{notation}
We will identify the dual space $\dB^\ast$ with a dense subspace inside $\mathcal{H}^2(\bC^4) \otimes \Lambda^2(\bR^4)$.

For a linear functional $f \in \dB^\ast$, we will identify it with an element $y$ in $\mathcal{H}^2(\bC^4) \otimes \Lambda^2(\bR^4)$ and denote the pairing $(x, y)_\sharp = f(x),\ x \in \dB$.

Define $\tilde{\mu}_{\theta}$, a product Wiener measure on $\dB$ with variance $1/\theta$, from Definition \ref{d.wm.1} and Proposition \ref{p.w.1}, which makes $(\dB, \tilde{\mu}_{\theta})$ a probability space. The $\sigma$-algebra is the Borel field in $\dB$, which is equipped with a norm in Equation (\ref{e.n.5}). Hence, $(\cdot, y)_\sharp$ is a Gaussian random variable on the Banach space and its variance is given by $\parallel y \parallel^2/\theta$.
\end{notation}

\section{Definition of the Yang-Mills Path Integral in abelian case}\label{s.fi}

We will now go back to Expression \ref{e.ym.2}. Consider the abelian group $G = {\rm U}(1)$, and  $\mathfrak{u}(1) \cong \bR \otimes \sqrt{-1}$. Let $S$ be a compact flat rectangular surface $S$, with $C = \partial S$ as its boundary. Using Stokes Theorem, \beq \int_C \sum_{j=1}^3 A_j \otimes dx^j = \int_S dA , \nonumber \eeq
where $A = \sum_{j=1}^3 A_j \otimes dx^j$.
Thus, we will rewrite Expression \ref{e.ym.2} as \beq \frac{1}{Z}\int_A e^{\int_S dA \otimes i}e^{-|dA|^2/2} DA, \nonumber \eeq and make sense of it using Abstract Wiener measure.

We have to do a change of variables, i.e. $A \mapsto dA$. This is an injective map because of Proposition \ref{p.p.1}. Then we have \beq \frac{1}{Z}\int_{dA} e^{c\int_S dA \otimes i}e^{-\frac{1}{2}|dA|^2} \det(d^{-1}) D[dA], \nonumber \eeq for some undefined constant $\det(d^{-1})$. After `dividing away' this constant, we are left with \beq \frac{1}{\bar{Z}}\int_{dA} e^{c\int_S dA \otimes i}e^{-\frac{1}{2}|dA|^2} D[dA], \label{e.p.1} \eeq which we can define as a Gaussian integral, for some new normalization constant \beq \bar{Z} = \int_{dA} e^{-\frac{1}{2}|dA|^2}  D[dA] .\nonumber \eeq Note that we introduced a coupling constant $c$.

Consider a 2-dimensional Euclidean space $\bR^2$. Then,
\begin{align*}
\int_S dA = \int_{\bR^2} dA \cdot1_S = \left\langle dA, 1_S \right\rangle.
\end{align*}
Here, $\langle \cdot, \cdot \rangle$ is the $L^2$ inner product on the space of Lebesgue integrable functions on $\bR^2$. Thus, the Yang-Mills path integral becomes \beq \frac{1}{\int_{dA} e^{-|dA|^2/2}D[dA]}\int_{dA} e^{\sqrt{-1}\langle dA, 1_S \rangle}e^{-\frac{1}{2}|dA|^2}D[dA]. \nonumber \eeq

This is a Gaussian integral of the form given in Equation (\ref{e.g.2}), hence we can define the path integral as $\exp\left[ -|1_S|^2/2\right] = \exp\left[ -|S|/2\right]$, whereby $|S|$ is the area of the surface $S$. Refer to \cite{Levy} for details on the computation.

\begin{defn}
For any $f_a \in H^2(\bC^4)$, define \beq \fd \sum_{a=1}^3 f_a \otimes dx^a := \sum_{a=1}^3 [\fd_0 f_a] dx^0 \wedge dx^a + \sum_{1 \leq i < j \leq 3}[\fd_i f_j - \fd_j f_i ]dx^i \wedge dx^j. \nonumber \eeq
\end{defn}

\begin{rem}
Just as we have the differential operator $d$ acting on 1-forms in $\mathcal{S}_\kappa(\bR^4) \otimes \Lambda^1(\bR^3)$, the analogous operator will be $\fd$ acting on $\Lambda^1(\bR^3)$-valued complex functions, which are in $H^2(\bC^4) \otimes \Lambda^1(\bR^3)$.
\end{rem}

Now $\partial_a [H_{p_r}(\kappa x)] = \kappa (\partial_a H_{p_r})(\kappa x)$. From the way we defined $\fd_a$, we see that \beq\Psi_\kappa (d f) = \kappa\fd (\Psi_\kappa [f]),\ f \in \mathcal{S}_\kappa(\bR^4) \otimes \Lambda^1(\bR^3). \label{e.k.1} \eeq Using the isometry of $\Psi_\kappa$, we have
\begin{align}
\Bigg\langle d\sum_{1\leq a \leq 3}&  f_a \otimes dx^a, d\sum_{1\leq b\leq 3} g_b \otimes dx^b \Bigg\rangle \nonumber\\
&= \kappa^2\left\langle \fd\sum_{1\leq a \leq 3} \Psi_\kappa[f_a] \otimes dx^a, \fd\sum_{1\leq b \leq 3} \Psi_\kappa[g_b] \otimes dx^b \right\rangle. \label{e.in.1}
\end{align}
This gives us our first renormalization rule: for each $\fd_a$, we assign a factor $\kappa$ to it. In other words, $\fd_a \mapsto \kappa \fd_a$, under renormalization.

\begin{rem}
Rescaling of the $\epsilon$-lattice to be a unit lattice in Remark \ref{r.c.3} will also lead to renormalization rules. Because of this scaling, the fields in consideration have to be rescaled in the following manner, i.e. \beq A(\epsilon x) \longmapsto \epsilon^{-a}A\left(x \right), \label{e.bab.1} \eeq for $x$ in the unit lattice. Note that $a$ is some power, dependent on the dimension $d$ of the lattice. In \cite{propagator1}, $a = (d-2)/2$.

Its derivatives will also transform. For example, in \cite{volumeI}, if we write $\phi(x) = \epsilon^{-(d-2)/2}\phi'(x/\epsilon)$, then the derivative will transform like \beq (\partial_\mu \phi)(x) = \epsilon^{-d/2}(\partial_\mu\phi')(x/\epsilon). \nonumber \eeq As a consequence, any renormalization or scaling will affect the action and hence the propagator. See \cite{volumeI}.
\end{rem}

Balaban defined a finite dimensional integral on a finite gauge lattice of spacing $\epsilon$. Our analogous procedure will be to replace Expression \ref{e.p.1} with \beq \frac{1}{\bar{Z}}\int_{\{dA: A \in \mathcal{S}_\kappa(\bR^4)\otimes \Lambda^1(\bR^3)\}} e^{c\int_S dA \otimes i}e^{-\frac{1}{2}|dA|^2} D[dA], \label{e.p.2} \eeq for some normalization constant $\bar{Z}$. The space $\mathcal{S}_\kappa(\bR^4)\otimes \Lambda^1(\bR^3)$ is dependent on $\kappa$, whereby each $\mathcal{S}_\kappa(\bR^4)\otimes \Lambda^1(\bR^3)$ is analogous to a finite lattice.

Balaban will use a renormalization flow to map a finite lattice gauge with spacing $\epsilon$ to a finite lattice gauge of unit spacing. The Segal Bargmann Transform will define a renormalization flow,
\beq \{\Psi_\kappa: \mathcal{S}_\kappa(\bR^4)\otimes \Lambda^1(\bR^3)\longrightarrow H^2(\bC^4)\otimes \Lambda^1(\bR^3) \}_{\kappa>0}. \nonumber \eeq Using Equation (\ref{e.in.1}), the Yang-Mills action will now become \beq
\langle dA, dA \rangle =
\kappa^2\int_{\bC^4}\sum_{i=1}^3 |\fd_0 \Psi_\kappa[A_i]|^2 + \sum_{1\leq i < j \leq 3}|\fd_i\Psi_\kappa[A_j] - \fd_j \Psi_\kappa[A_i]|^2\ d\lambda_4, \nonumber \eeq whereby $A = \sum_{i=1}^3 A_i \otimes dx^i \in \mathcal{S}_\kappa(\bR^4)\otimes \Lambda^1(\bR^3)$.
Note that \beq \fd_0 \Psi_\kappa[A_i] \in \fd_0 H^2(\bC^4)\quad {\rm and} \quad \fd_i\Psi_\kappa[A_j] - \fd_j \Psi_\kappa[A_i] \in H^2(\bC^4). \nonumber \eeq
In defining this action, we applied the renormalization rule given by Equation (\ref{e.k.1}).

Using the renormalization rule and Segal Bargmann Transform as a renormalization flow and isometry, Expression \ref{e.p.2} becomes \beq \frac{1}{\bar{Z}}\int_{\fd \tilde{A}\in \mathbb{H}} e^{c\kappa\int_S \fd \tilde{A} \otimes i}e^{-\frac{\kappa^2}{2}|\fd \tilde{A}|^2} D[\fd \tilde{A}],\ \tilde{A} = \Psi_\kappa[A], \nonumber \eeq and \beq \bar{Z} = \int_{\fd \tilde{A}\in \mathbb{H}}e^{-\frac{\kappa^2}{2}|\fd \tilde{A}|^2} D[\fd \tilde{A}]. \nonumber \eeq

We are going to replace $ \fd\tilde{A}$ with \beq B \in \mathbb{H} =  \left\{\left[\fd_0 H^2(\bC^4)\right] \otimes \left[\ast\Lambda^2(\bR^3)\right]\right\} \oplus \left\{H^2(\bC^4) \otimes \Lambda^2(\bR^3) \right\}, \nonumber \eeq so our expression now becomes  \beq \frac{1}{\bar{Z}}\int_{B \in \mathbb{H}} e^{c\kappa\int_S B} e^{-\kappa^2|B|^2/2} DB. \label{e.n.4} \eeq The parameter $\kappa>0$ will tell us that we are mapping $\mathcal{S}_\kappa(\bR^4)\otimes \Lambda^1(\bR^3)$ into $H^2(\bC^4) \otimes \Lambda^1(\bR^3)$. This means that the path integral is not defined on the space of $\Lambda^2(\bR^4)$-valued real functions in $\bR^4$, but rather, over the space of $\Lambda^2(\bR^4)$-valued complex functions in $\bC^4$. Because of the term $|\kappa B|^2$, we will define the path integral in Expression \ref{e.n.4}, over in $\dB$, equipped with a product Wiener measure $\tilde{\mu}_{\kappa^2}$, variance $1/\kappa^2$.

\begin{defn}
Recall that for any $f \in \mathcal{H}^2(\bC^4)_\bC$, $\chi_w(z) \equiv e^{\bar{w}\cdot z}$ is the unique vector such that $\langle f, \chi_w \rangle = f(w)$ for any $w \in \bC^4$. Because $(\cdot, \chi_w)_\sharp \in B(\bC^4)_\bC^\ast$, we see that it is also in $B_0(\bC^4)_\bC^\ast$ and $[\fd_0 B_0(\bC^4)]_\bC^\ast$.

Define the following linear functional  \beq \tilde{\xi}_{ab}(w):= (-1)^{ab}\chi_w \otimes dx^a \wedge dx^b,\ 0 \leq a < b \leq 3. \nonumber \eeq We will identify it with a complex random variable $(\cdot, \tilde{\xi}_{ab}(w))_\sharp$ in the dual space $\dB^{\ast}_{\mathbb{C}}$.
\end{defn}

\begin{rem}
Clearly,
\begin{align}
(\cdot, \tilde{\xi}_{ab}(w))_\sharp:\ f =& \sum_{i=1}^3f_{0i}dx^0 \wedge dx^i + \sum_{1 \leq i < j \leq 3} (-1)^{ij}f_{ij}dx^i \wedge dx^j \in H^2(\bC^4)\otimes \Lambda^2(\bR^4) \nonumber \\
\longmapsto& \langle f,\tilde{\xi}_{ab}(w) \rangle = f_{ab}(w). \label{e.f.2}
\end{align}

Or, \beq \langle \fd A, \tilde{\xi}_{ab}(x) \rangle = (-1)^{ab}(\fd_a A_b - \fd_b A_a)(x). \nonumber \eeq
\end{rem}


Our final step is to rewrite $c\kappa\int_S B$ as a random variable. We need to define a renormalization transformation, which sends \beq x \in \bR^4\quad {\rm to}\quad \kappa x/2 \in \bR^4 \hookrightarrow \bC^4,\nonumber \eeq hence embedding $\bR^4 \hookrightarrow \bC^4$. As a consequence, given the curve $C \equiv \partial S \in \bR^4$, we scale it by a factor of $\kappa/2$ inside $\bC^4$; likewise the surface $S$ is scaled by a factor of $\kappa/2$.

\begin{rem}
In \cite{volumeI}, the renormalization transformation will affect the action. In our case, the scaling affects only the integrand.
\end{rem}

Let $\sigma: [0,1]^2\equiv I^2 \rightarrow \bR^4$ be any parametrization of the rectangular surface $S$. Define \beq \tilde{\nu}_S^\kappa = \kappa\int_{I^2}dsdt\ \sum_{0\leq a < b \leq 3} \frac{\kappa^2}{4}|J_{ab}^\sigma|(s,t)\tilde{\xi}_{ab}(\kappa\sigma(s,t)/2) , \nonumber \eeq whereby $\tilde{\xi}_{ab}(w)$ and $|J_{ab}^\sigma|$ are defined in Equation (\ref{e.f.2}) and Definition \ref{d.r.1} respectively. The factor $\kappa^2/4$ comes from embedding the surface $S \subset \bR^4$ inside $\bC^4$, but scaled by the renormalization transformation $x \in \bR^4 \mapsto \kappa x/2 \in \bC^4$.

Note that $(\cdot, \tilde{\nu}_S^\kappa)_\sharp$ is a random variable on the probability space $(\dB, \tilde{\mu}_{\kappa^2})$. For $B = \sum_{0\leq a < b \leq 3} (-1)^{ab}B_{ab} \otimes dx^a \wedge dx^b$,
\begin{align*}
(\cdot, \tilde{\nu}_S^\kappa)_\sharp: B \in \dB &\longmapsto \kappa\int_{I^2}dsdt\ \sum_{0\leq a < b \leq 3} \frac{\kappa^2}{4}|J_{ab}^\sigma|(s,t)(B, \tilde{\xi}_{ab}(\kappa\sigma(s,t)/2))_\sharp \\
&\equiv \kappa \int_{I^2}dsdt\ \sum_{0\leq a < b \leq 3} \frac{\kappa^2}{4}|J_{ab}^\sigma|(s,t)B_{ab}(\kappa\sigma(s,t)/2) \\
&= \kappa \int_{\kappa S/2} B.
\end{align*}
It is a Gaussian random variable, with variance $\parallel\tilde{\nu}_S^\kappa \parallel^2/\kappa^2$.
This follows from Corollary \ref{c.w.5}, which implies $\fd_0 B_0(\bC^4)$ contains a complete set of basis in $\mathcal{H}^2(\bC^4)$.

\begin{rem}
Note that $\parallel\tilde{\nu}_S^\kappa \parallel^2 = \langle \tilde{\nu}_S^\kappa, \tilde{\nu}_S^\kappa \rangle$.
\end{rem}

Refer to Notation \ref{n.n.5}. Instead of using $\tilde{\xi}_{ab}(w)$, we replace it with $\xi_{ab}(w)$, where $\xi_{ab}(w) := \psi(w) \tilde{\xi}_{ab}(w)$. Thus define \beq \nu_S^\kappa = \kappa\int_{I^2}ds dt\ \sum_{0\leq a < b \leq 3} \frac{\kappa^2}{4}|J_{ab}^\sigma|(s,t)\xi_{ab}(\kappa\sigma(s,t)/2) . \label{e.n.3} \eeq The factor $\psi_w$ can be regarded as a form of renormalization factor, necessary when we take the limit as $\kappa$ goes to infinity.

\begin{rem}
\begin{enumerate}
  \item In the definition of the Chern Simons path integral in \cite{CSLim01}, this factor $\psi_w$ was also used.
  \item In the earlier part of this article, we explained that the propagator will blow up as $\kappa$ goes to infinity. Hence, we multiply by a factor $\psi_w$, which is strictly less that 1. This is analogous to impose a cut-off limit in the momentum space, when one uses Fourier Transform instead.
\end{enumerate}
\end{rem}

\begin{defn}\label{d.ym.1}(Definition of Yang-Mills path integral in the abelian case)\\
Let $\nu_S^\kappa$ be given by Equation (\ref{e.n.3}). For some fixed constant $c > 0$, we now define Expression \ref{e.n.4}, in the Banach space $\dB$, as
\begin{align}
\mathbb{E}_{{\rm YM}} \left[ \exp\left[ ic\left( \cdot, \nu_S^\kappa \right)_\sharp \right] \right]
:=& \int_{B \in \dB} \exp\left[  i c\left(B, \nu_S^\kappa \right)_\sharp \right] d\tilde{\mu}_{\kappa^2}(B).
\label{e.a.1}
\end{align}
\end{defn}

\begin{rem}\label{r.c.2}
\begin{enumerate}
  \item In \cite{peskin1995introduction}, they introduced the coupling constant $c$ in the exponent of the line integral over $C$.
  \item Notice that this is an infinite dimensional integral, which is analogous to a finite dimensional integral defined on a finite lattice. In lattice gauge theory, one has to take the spacing $\epsilon$ to go down to zero. This is analogous to us taking $\kappa$ going to infinity, so that the correlation between distinct points go down to zero. 
  \item To prove ultraviolet stability, Balaban defined a sequence of partition functions, each partition function defined on a lattice with decreasing spacing. See \cite{1985d}. Using the renormalization flow, we also defined a sequence of path integrals, indexed by $\kappa$.
\end{enumerate}
\end{rem}

\begin{thm}\label{t.m.1}
Consider the abelian Lie group ${\rm U}(1)$. Then,
\begin{align}
\int_{B \in \dB}& \exp\left[ i\sum_{0\leq a<b  \leq 3}\int_{I^2}\frac{\kappa^2}{4}|J_{ab}^{\sigma}|(s,t) \left(B,\xi_{ab}(\kappa\sigma(s,t)/2)\right)_\sharp ds dt\right] d\tilde{\mu}_{\kappa^2}(B) \nonumber \\
=& \exp\left[ -\frac{1}{2}\left\Vert\sum_{0\leq a<b  \leq 3}\int_{I^2} dsdt\ \frac{\kappa}{4}|J_{ab}^{\sigma}|(s,t)\xi_{ab}(\kappa\sigma(s,t)/2)  \right\Vert^2\right] \nonumber \\
\longrightarrow& \exp\left[ -\frac{1}{8}\int_S d\rho \right], \label{e.qc.2}
\end{align}
as $\kappa$ goes to infinity. See Equation (\ref{e.a.3}).
\end{thm}

\begin{proof}
Recall $\psi_w $ from Notation \ref{n.n.5}, and $\xi_{ab}(w) = (-1)^{ab}\psi_w\chi_w \otimes dx^a \wedge dx^b$. For $v,w \in \bR^4 \subset \bC^4$, \beq \langle \psi_w\chi_w , \psi_v \chi_v \rangle = \psi_w\psi_v e^{w\cdot v} = \frac{1}{2\pi}e^{-|w-v|^2/2}. \nonumber \eeq
Hence \beq \langle  \xi_{ab}(w),  \xi_{ab}(\hat{w}) \rangle = \frac{1}{2\pi}e^{-|w - \hat{w}|^2/2}. \label{e.i.5} \eeq

Write \beq \eta_S = \sum_{0\leq a<b  \leq 3}\int_{I^2} dsdt\ \frac{\kappa}{4}|J_{ab}^{\sigma}|(s,t)\xi_{ab}(\kappa\sigma(s,t)/2). \nonumber \eeq
Because we are using the Wiener measure $\tilde{\mu}_{\kappa^2}$, thus $\kappa(\cdot, \eta_S)_\sharp/\parallel \eta_S \parallel$
is equal to the standard normal, in distribution.

Write $I^4 = I^2 \times I^2$. Now,
\begin{align*}
&\left\Vert\sum_{0\leq a<b \leq 3}\int_{I^2}dsdt\ |J_{ab}^\sigma|(s,t)\frac{\kappa}{4}\xi_{ab}(\kappa\sigma(s,t)/2) \right\Vert^2 \\
&=\sum_{0\leq a<b \leq 3}\int_{I^4} ds dt d\bar{s}d\bar{t}\ \frac{\kappa^2}{16}|J_{ab}^{\sigma}|(s,t)|J_{ab}^{\sigma}|(\bar{s},\bar{t})\left\langle \xi_{ab}(\kappa\sigma(s,t)/2),  \xi_{ab}(\kappa\sigma(\bar{s}, \bar{t})/2) \right\rangle \\
&= \frac{1}{4(2\pi)}\sum_{0\leq a<b \leq 3}\int_{I^2 \times I^2} ds dt d\bar{s}d\bar{t}\ \frac{\kappa^2}{4}|J_{ab}^{\sigma}|(s,t)|J_{ab}^{\sigma}|(\bar{s},\bar{t})e^{-\kappa^2|\sigma(s,t) - \sigma(\bar{s}, \bar{t})|^2/8}  \\
&\longrightarrow_{\kappa \rightarrow \infty} \frac{1}{4} \sum_{0\leq  a<b \leq 3}\int_{I^2} ds dt\ |J_{ab}^\sigma|(s,t) \rho_\sigma^{ab}(\sigma(s,t)),
\end{align*}
the limit follows from Corollary \ref{c.w.4}.

For $\theta > 0$,
\beq \frac{\theta}{\sqrt{2\pi}}\int_{\bR} e^{\alpha x}e^{-\theta^2 x^2/2}dx = e^{\alpha^2/2\theta^2}.\label{e.g.2} \eeq Therefore, the path integral is equal to
\begin{align*}
\exp&\left[ -\frac{1}{2}\left\Vert\sum_{0\leq a<b \leq 3}\int_{I^2}dsdt\ |J_{ab}^\sigma|(s,t)\frac{\kappa}{4}\xi_{ab}(\kappa\sigma(s,t)/2) \right\Vert^2\right] \\
&\longrightarrow \exp\left[ -\frac{1}{8} \sum_{0\leq  a<b \leq 3}\int_{I^2} ds dt\ |J_{ab}^\sigma|(s,t) \rho_\sigma^{ab}(\sigma(s,t))\right] = e^{-\frac{1}{8}\int_S d\rho },
\end{align*}
as $\kappa$ goes to infinity.
\end{proof}

In the case of dimension $d = 4$, if one consider the non-abelian Yang-Mills gauge theory on a lattice gauge, the idea that the coupling constant should depend on $\epsilon$, and it goes down to 0 as $\epsilon$ goes down to 0, is known as asymptotic freedom. Refer to \cite{Gross01}.

\begin{rem}
Balaban's work does not incorporate asymptotic freedom. See \cite{1987-99_RG-01}.
\end{rem}

In this article, we will define asymptotic freedom, by setting the coupling constant $c$ to be inversely proportional to $\kappa$, i.e. $c = 1/\kappa$. This means that as $\kappa$ goes to infinity, the coupling constant goes down to 0.

If we set the coupling constant $c = 1/\kappa$ in RHS of Equation (\ref{e.a.1}), then a direct computation will show that
\begin{align}
&\int_{\dB} \exp\left[  \frac{i}{\kappa}\left(\cdot, \nu_S^\kappa  \right)_\sharp \right] d\tilde{\mu}_{\kappa^2} \nonumber \\
&= \int_{\dB} \exp\left[ i\sum_{0\leq a<b  \leq 3}\int_{I^2}\frac{\kappa^2}{4}|J_{ab}^{\sigma}|(s,t) \left(\cdot,\xi_{ab}(\kappa\sigma(s,t)/2) \right)_\sharp ds dt\right] d\tilde{\mu}_{\kappa^2}. \label{e.c.3}
\end{align}

\begin{rem}\label{r.c.1}
In \cite{Gross01}, Gross added $1/c^2$ to the Yang-Mills action, $c$ is the coupling constant.  But here, adding a factor $\kappa^2$ in front of the Yang-Mills action, is due to renormalization, and not due to asymptotic freedom.
\end{rem}

From the discussion in Appendix \ref{ss.si}, we note that $\int_S d\rho$ is the area of the surface $S$. Thus, the limit of the path integral given in Equation (\ref{e.c.3}) as $\kappa \rightarrow \infty$, will give us the Area Law formula, provided we assume asymptotic freedom holds, i.e. $c = 1/\kappa$. The Area Law formula is used to explain why the strong force gets stronger as distance increases, thus we are not able to find isolated quarks in nature. See \cite{Nair}.

But we know in nature that there is no binding quarks for electromagnetic force. Suppose the coupling constant $c$ is fixed and is independent of $\kappa$. Then, the preceding calculations can be modified accordingly to show that the limit will give us 0, thus showing that in the ${\rm U}(1)$ gauge theory, the Area Law formula does not hold.

\begin{cor}
Consider the abelian Lie group ${\rm U}(1)$. Define the Yang-Mills path integral using RHS of Equation (\ref{e.a.1}). If the coupling constant $c$ is fixed, then we have
\begin{align}
\int_{B \in \dB} \exp\left[  c\left(B, \nu_S^\kappa \otimes i  \right)_\sharp \right] d\tilde{\mu}_{\kappa^2}(B)
\longrightarrow& 0, \nonumber
\end{align}
as $\kappa$ goes to infinity.
\end{cor}

\section{Concluding remarks}\label{s.cr}

In \cite{MR2076918}, Hahn used White Noise analysis, to make sense of the Wilson Loop observable, but using the Chern-Simons action instead. In the completion of the Hilbert space containing (real) Schwartz functions, one obtains the space containing generalized functions. To make sense of the Chern-Simons path integral to obtain the linking number for link, he had to `smear' the link. In \cite{CSLim01}, we used the Segal Bargmann approach and defined the Chern-Simons path integral using holomorphic functions, which motivate us to use the same construction for Yang-Mills path integral. Indeed, we obtained the same linking number for links, but without having to `smear' the link.

Instead of making sense of the Wilson Loop observable (for a given loop $C$) given by Expression \ref{e.ym.2}, we made sense of Expression \ref{e.n.4} in place of it, for a surface $S$ whose boundary is $C$. This path integral is made sense of over the quantum space $\dB$, which consists of $\Lambda^2(\bR^4)$-valued holomorphic functions over in $\mathbb{C}^4$. Just as in the Chern-Simons case, when one tries to complete the Hilbert space containing $\Lambda^2(\bR^4)$-valued Schwartz functions over in $\bR^4$, one will end up with generalized functions.

Recall in Section \ref{s.gm}, we explained that one has to deal with a Dirac Delta functional, if we insist on working in the Hilbert space containing real Schwartz functions. The problem comes when we try to write $\int_S dA$ in the form $\langle dA, \mu_S \rangle$, when both $dA$ and $\mu_S$ are generalized functions. To remedy this, we would have to `smear' the surface $S$, and approximate $\mu_S$ with a sequence of $\Lambda^2(\bR^4)$-valued Schwartz functions $\{\mu_S^\kappa\}_{\kappa > 0}$, which we prefer to avoid. This will be similar to Hanh's approach in defining the Chern-Simons path integral. By taking the limit as $\kappa \rightarrow \infty$, of a sequence of path integrals defined using $\mu_S^\kappa$, one will still obtain the Area Law formula, after applying renormalization rules and asymptotic freedom, as discussed throughout this article.

Both approaches will give us the same result, even though the construction is entirely different and not equivalent. But we firmly believe that the Segal Bargmann approach is easier and it avoids technical difficulties like dealing with a Dirac Delta functional and `smearing' the surface. Furthermore, when one considers the non-abelian Yang-Mills path integral in the sequel to this article, the triple and quartic product terms in the Yang-Mills action can be easily define using holomorphic functions in the quantum space $\dB$, whereas products of distributions are ill-defined.

\appendix

\section{Surface Integrals}\label{ss.si}

Let $S$ be a compact surface embedded in $\bR^4$. Suppose $\sigma\equiv ( \sigma_0, \sigma_1, \sigma_2, \sigma_3): [0,1]^2 \equiv I^2 \rightarrow \bR^4$ is its parametrization. Here, $\sigma' = \partial \sigma/\partial s$ and $\dot{\sigma} = \partial \sigma/\partial t$. Write $x \equiv (x_0, x_1, x_2, x_3)$.

\begin{defn}\label{d.r.1}
Define Jacobian matrices,
\begin{align}
J_{ij}^\sigma(s,t) = \left(
               \begin{array}{cc}
                 \sigma_i'(s,t) & \dot{\sigma}_i(s,t) \\
                 \sigma_j'(s,t) & \dot{\sigma}_j(s,t) \\
               \end{array}
             \right), \nonumber
\end{align}
and write $|J^\sigma_{ij}| = |\det{J^\sigma_{ij}}|$ and $W_{ab}^{ cd} := J_{cd}^\sigma J_{ab}^{\sigma,-1}$, $a, b, c, d$ all distinct. Note that $W_{cd}^{ab} = (W_{ab}^{cd})^{-1}$.

For $a,b, c, d$ all distinct, define $\rho_\sigma^{ab}: I^2 \rightarrow \bR$ by
\begin{align}
\rho_\sigma^{ab} =& \frac{1}{\sqrt{\det\left[ 1+ W_{ab}^{cd,T}W_{ab}^{cd}\right]}}. \nonumber
\end{align}
\end{defn}

\begin{lem}\label{l.w.5}
Let $\vec{x}$, $\vec{y} $ be orthogonal directional vectors in $\bR^4$. Suppose $S$ is a flat compact rectangular surface parametrized by $\sigma: (s,t) \in I^2 \mapsto \vec{a} + s \vec{x} + t\vec{y}$, for some vector $\vec{a} \in \bR^4$.

Assume that $\sigma_a', \dot{\sigma}_a, \sigma_b', \dot{\sigma}_b \neq 0$ and let $\sigma(\bp) = \mathbf{x}$, $\bp = (s_0, t_0)$ lie in the interior of $I^2$. We have for distinct $a, b, c, d$,
\beq \int_{U}  \frac{\kappa^2}{4}e^{-\kappa^2|\sigma(s,t)-\mathbf{x}|^2/8} |\sigma'_a\dot{\sigma}_b - \sigma_b'\dot{\sigma}_a|(s,t) ds dt
= \frac{2\pi}{\sqrt{\det\left[1 + W_{ab}^{cd,T}W_{ab}^{cd}(\bp)\right]}} +O(e^{-\kappa^2 \bar{C}}), \nonumber \eeq
for some positive constant $\bar{C}$ that depends on $\mathbf{x}$ in the surface $S$, independent of $\kappa$.
\end{lem}

\begin{proof}
Write $\underline{s}=s-s_0$, $\underline{t}=t-t_0$ and $J_{ab} \equiv J_{ab}^\sigma(\bp)$. Let $\sigma_i(s_0, t_0) = \mathbf{x}_i$. Using Taylor's theorem,
\beq \sigma_i(s,t)= \mathbf{x}_i + \underline{s} \sigma_i'(\bp) + \underline{t}\dot{\sigma}_i(\bp) . \nonumber \eeq

Use a transformation \beq
\left(
  \begin{array}{c}
    u \\
    v \\
  \end{array}
\right) =
\left(
  \begin{array}{c}
    \underline{s}\sigma_a'(\bp) + \underline{t}\dot{\sigma}_a(\bp) \\
    \underline{s}\sigma_b'(\bp) + \underline{t}\dot{\sigma}_b(\bp) \\
  \end{array}
\right) = J_{ab}
\left(
  \begin{array}{c}
    \underline{s} \\
    \underline{t} \\
  \end{array}
\right)
. \nonumber \eeq
And let $0$ be in the interior of $V$, which is the range of $I^2$ under this transformation. Then, we have
\begin{align*}
\int_{I^2} \frac{\kappa^2}{4}& e^{-\kappa^2|\sigma(s,t)-\mathbf{x}|^2/8} |\sigma'_a\dot{\sigma}_b - \sigma_b'\dot{\sigma}_a|(s,t)\ ds dt  \\
=& \int_{I^2} \frac{\kappa^2}{4}e^{-\kappa^2 \left[|J_{ab}(\underline{s},\underline{t})^T|^2 + |J_{cd}(\underline{s},\underline{t})^T|^2\right]/8}  |\sigma'_a\dot{\sigma}_b - \sigma_b'\dot{\sigma}_a| (\bp)\  ds dt  \\
=& \int_V \frac{\kappa^2}{4}e^{-\kappa^2\left[(u^2+v^2) \right]/8} e^{-\kappa^2 \left[ |J_{cd}J_{ab}^{-1}(u,v)^T|^2  \right]/8}\  du dv  \\
=& \int_V \frac{\kappa^2}{4}e^{-\kappa^2  \left[\langle (1 + J_{ab}^{T,-1} J_{cd}^T J_{cd}J_{ab}^{-1})(u,v)^T, (u,v)^T \rangle \right] /8}\  du dv .
\end{align*}

Do another substitution, $\kappa u/2 \mapsto u$, $\kappa v/2 \mapsto v$. If $0 \in V = [\underline{\delta}, \overline{\delta}] \times [\underline{\epsilon}, \overline{\epsilon}] $, then the transformed region becomes \beq V_\kappa := [\kappa\underline{\delta}/2, \kappa\overline{\delta}/2] \times [\kappa\underline{\epsilon}/2, \kappa\overline{\epsilon}/2]. \nonumber \eeq Let \beq D = \frac{2\pi}{\sqrt{\det\left[1 + J_{ab}^{T,-1} J_{cd}^T J_{cd}J_{ab}^{-1}\right]}}= \frac{2\pi|J_{ab}|}{\sqrt{\det\left[J_{ab}^TJ_{ab} +  J_{cd}^T J_{cd}\right]}}. \nonumber \eeq

So the integral now becomes
\begin{align*}
\int_{V_\kappa}& e^{- \left[\langle (1 + J_{ab}^{T,-1} J_{cd}^T J_{cd}J_{ab}^{-1})(u,v)^T, (u,v)^T \rangle \right] /2}\  du dv \\
=& \int_{\bR^2}e^{- \left[\langle (1 + J_{ab}^{T,-1} J_{cd}^T J_{cd}J_{ab}^{-1})(u,v)^T, (u,v)^T \rangle \right] /2}\  du dv  \\
&- \int_{\bR^2\setminus V_\kappa }e^{- \left[\langle (1 + J_{ab}^{T,-1} J_{cd}^T J_{cd}J_{ab}^{-1})(u,v)^T, (u,v)^T \rangle \right] /2}\  du dv = D + O(e^{-\kappa^2 \bar{C}}),
\end{align*}
for positive constant $\bar{C}$ that depends only on the point $\mathbf{x}$. Because 0 is in the interior of $V \subset \bR^2$, we can choose $\bar{C} = \frac{1}{4}{\rm min}\{|\underline{\delta}|^2, \overline{\delta}^2, |\underline{\epsilon}|^2, \overline{\epsilon}^2\}>0$.
\end{proof}

\begin{rem}
It is crucial that we assume that $\mathbf{x}$ lies in the interior, as the result is no longer true if $\mathbf{x}$ lies on the boundary of $S$.
\end{rem}

\begin{cor}\label{c.w.4}
Refer to Definition \ref{d.r.1}. Let $\sigma$ be any parametrization for a compact flat rectangular surface $S \subset \bR^4$. We have
\begin{align}
\sum_{0 \leq a<b \leq 3}&\frac{\kappa^2}{4} \int_{I^2 \times I^2} ds dt d\bar{s}d\bar{t}\ e^{-\kappa^2|\sigma(s,t)-\sigma(\bar{s},\bar{t})|^2/8} |J_{ab}^\sigma|(s,t)|J_{ab}^\sigma|(\bar{s},\bar{t}) \nonumber \\
\longrightarrow& 2\pi \sum_{0 \leq a<b \leq 3}\int_{I^2}\rho_\sigma^{ab}(s,t)|J_{ab}^\sigma|(s,t)\label{e.a.3}
\ ds dt,
\end{align}
as $\kappa$ goes to infinity. We will write
\beq 2\pi \sum_{0 \leq a<b \leq 3}\int_{I^2}\rho_\sigma^{ab}(s,t)|J_{ab}^\sigma|(s,t)
\ ds dt =: 2\pi \int_S d\rho,\nonumber \eeq
which is independent of the parametrization $\sigma$ used.
\end{cor}

\begin{proof}
Write \beq f_\kappa(s, t) = \sum_{0 \leq a<b \leq 3}\frac{\kappa^2}{4} \int_{I^2 } d\bar{s} d\bar{t} \ e^{-\kappa^2|\sigma(s,t)-\sigma(\bar{s},\bar{t})|^2/8} |J_{ab}^\sigma|(s,t)|J_{ab}^\sigma|(\bar{s},\bar{t}). \nonumber \eeq From the previous lemma, we see that \beq \lim_{\kappa \rightarrow \infty}f_{\kappa}(s, t) = 2\pi \sum_{0 \leq a<b \leq 3}\rho_\sigma^{ab}(s,t)|J_{ab}^\sigma|(s,t) \nonumber \eeq pointwise, for $(s,t)$ in the interior of $I^2$. Since $|f_\kappa| \leq M$ for some constant $M$, we can  apply the Dominated Convergence Theorem and see that
\begin{align*}
\lim_{\kappa \rightarrow \infty}\sum_{0 \leq a<b \leq 3}&\frac{\kappa^2}{4} \int_{I^2 \times I^2} ds dt d\bar{s}d\bar{t}\ e^{-\kappa^2|\sigma(s,t)-\sigma(\bar{s},\bar{t})|^2/8} |J_{ab}^\sigma|(s,t)|J_{ab}^\sigma|(\bar{s},\bar{t}) \\
=& \int_{I^2}\lim_{\kappa \rightarrow \infty} f_\kappa(s,t)\ ds dt = 2\pi \int_S d\rho.
\end{align*}
\end{proof}

\begin{rem}
Clearly, the quantity $\int_S d\rho$ is independent of any orthonormal basis used to span $\bR^4$. Indeed, this formula can be used for any compact surface $S$ and it is straight forward to show that $\int_S d\rho$ is equal to the area of the surface $S$.
\end{rem}

\end{document}